\documentclass[12pt]{article}

\usepackage{amsmath,amsthm,amsfonts,amssymb,amscd}
\newtheorem {Lemma}{Lemma}[section]
\newtheorem {Theorem}{Theorem}[section]

\newtheorem {Conjecture}{Conjecture}[section]
\numberwithin{equation}{section}

\allowdisplaybreaks [4]
\headsep =
0.0in \headheight = 0.0in \topmargin = 0.3in

\usepackage[T1]{fontenc}
\usepackage[utf8]{inputenc}
\usepackage{authblk}

\begin{document}

\title{A proof of a conjecture on the distance spectral radius}
%\author[1]{Yanna Wang\footnote{E-mail: yi.fuxue@163.com}, Bo Zhou\footnote{Coresponding author. E-mail: zhoubo@scnu.edu.cn}
%\\
%Basic Courses Department, Guang Dong Communication Polytechnic, \\
%Guangzhou 510650, P. R. China%\\
%\\
%School of Mathematical Sciences, South China Normal University, \\
%Guangzhou 510631, P. R. China
%}
\author[1]{Yanna Wang\footnote{E-mail: wangyn@gdcp.edu.cn}}

\author[2]{Bo Zhou\footnote{Coresponding author. E-mail: zhoubo@scnu.edu.cn}}

\affil[1]{Basic Courses Department, Guangdong Communication Polytechnic,
Guangzhou 510650, P. R. China}

\affil[2]{School of Mathematical Sciences, South China Normal University,
Guangzhou 510631, P. R. China}

\date{}

\maketitle

\begin{abstract}
A cactus is a connected graph in which any two cycles have at most one common vertex.
We determine the unique graph that maximizes the distance spectral radius over  all cacti with fixed numbers of vertices and cycles, and thus prove a conjecture on the distance spectral radius of cacti
in [S.S. Bose, M. Nath, S. Paul,
On the distance spectral radius of cacti,
Linear Algebra Appl. 437 (2012) 2128--2141]. We prove the result in the context of hypertrees.
\end{abstract}

\section{Introduction}

A (simple) hypergraph $G$ consists of a vertex set $V(G)$  and an edge set $E(G)$,  where every edge in $E(G)$ is a subset of $V(G)$ containing at least two vertices,
 see \cite{Ber}.  The rank of $G$ is the maximum cardinality of edges of $G$. For an integer $k\geq 2$, we say that  $G$ is $k$-uniform if every edge of $G$ contains exactly $k$ vertices. An ordinary (simple) graph is just a $2$-uniform hypergraph. For $u,v\in V(G)$, if they are contained in some edge of $G$, then we say that they are adjacent, or $v$ is a neighbor of $u$. For $u\in V(G)$, let $N_G(u)$ be the set of neighbors of $u$ in $G$ and $E_G(u)$ be the set of edges containing $u$ in $G$. The degree of a vertex $u$ in $G$, denoted by $\mbox{deg}_G(u)$, is $|E_G(u)|$.

For distinct vertices $v_0,\dots, v_{p}$ and distinct edges $e_1, \dots, e_p$ of $G$,
the alternating sequence %of vertices and edges
$(v_0,e_1,v_1,\dots,v_{p-1},e_p,v_p)$ such that $v_{i-1}, v_i\in e_i$ for $i=1,\dots,p$ and $e_i\cap e_j=\emptyset$ for $i,j=1,\dots,p$ with $j>i+1$
is a  loose path of $G$ from $v_0$ to $v_p$ of length $p$.
%The  vertices $v_0,\dots, v_{p}$ are called the core vertices of the  path.
If there is a  loose path from $u$ to $v$ for any $u,v\in V(G)$, then we say that $G$ is connected.

For distinct vertices $v_0,\dots, v_{p-1}$ and distinct edges $e_1, \dots, e_p$,
the alternating sequence %of  vertices and  edges
$(v_0,e_1,v_1,\dots,v_{p-1},e_p,v_0)$ such that $v_{i-1},v_i\in e_i$ for $i=1, \dots, p$ with $v_p=v_0$ and $e_i\cap e_j=\emptyset$ for $i,j=1,\dots,p$ with $\mid i-j \mid>1$ and $\{i,j\}\neq \{1,p\}$
is a  loose cycle of $G$ of length $p$. A hypertree is a connected hypergraph with no  loose cycles.

An edge $e$ of a hypergraph $G$ is called a pendant edge of $G$ at $v$ if $v\in e$, the degree of all vertices of $e$ except $v$ in $G$ is one, and $\mbox{deg}_G(v)>1$. A pendant vertex is a vertex of degree one.

Any  hypergraph $G$ corresponds naturally to a graph  $O_G$  with $V(O_G)=V(G)$
such that for $u,v\in V(O_G)$,  $\{u, v\}$ is an edge of $O_G$ if and only if $u$ and $v$ are in some edge of $G$. Obviously, an edge of $G$ with size $r$ corresponds naturally to a clique (maximal $2$-connected subgraph)  of $O_G$ with size $r$.
If $G$ is connected, then the distance between vertices  $u$ and $v$ in $G$ (or $O_G$), denoted by  $d_{G}(u,v)$,  is the length of a shortest  loose path connecting them in $G$. For some extremal spectral problems related to distance, we find hypergraph notation is  more convenient and effective.

Let $G$ be a connected hypergraph  on $n$ vertices. The distance matrix of $G$ is defined as $D(G)=(d_{G}(u,v))_{u, v\in V(G)}$.
The distance
spectral radius of $G$, denoted by $\rho(G)$, is the largest
eigenvalue of $D(G)$.
The  eigenvalues of distance matrices of graphs, arisen from a data communication problem studied by Graham and
Pollack \cite{GP} in 1971, have been studied extensively,
and particularly, the distance spectral radius received much attention, see the
 survey \cite{AH-2}. We mentioned that the distance spectral radius has also been used as a molecular descriptor, see \cite{BCM,GM}.
Watanabe et al.~\cite{WIS} studied spectral properties of the distance matrix of uniform hypertrees, generalizing some results by Graham and Pollak \cite{GP} and Sivasubramanian~\cite{Si}.
Lin  and Zhou~\cite{LZ1}
%and Lin et al.~\cite{LZL}
studied the distance spectral radius of  uniform hypergraphs and particularly, uniform hypertrees.
 %showed that $S_{n,k}$ and $P_{n,k}$ are the unique $k$-uniform hypertrees on $n$ vertices with minimum and maximum  distance spectral radius, respectively, and   determined the unique hypertree with maximum distance spectral radius among $k$-uniform hypertrees with fixed maximum degree.
%  studied the distance spectral radius of some  particular uniform hypertrees.
Wang  and Zhou~\cite{WZ2} studied the distance spectral radius of a hypergraph that is not necessarily uniform, and
%They proposed some  graft transformations that decrease or increase the distance  spectral radius of a hypergraph,
 determined the unique  hypertrees with  minimum and maximum  distance  spectral radius, respectively, among hypertrees on $n$ vertices with $m$ edges, where $1\leq m\leq n-1$, and also determined the unique  hypertrees with the first three smallest (largest, respectively)  distance  spectral radii among hypertrees on $n\ge 6$ vertices. In \cite{WZ3}, they made further efforts to identify extremal hypergraphs in some classes of hypergraphs with given parameters.

%We mention that Balaban et al.~\cite{BCM} proposed the use of the distance spectral radius
%of graphs as a molecular descriptor,
%and it was successfully
%used to make inferences about the extent of branching and boiling points of
%alkanes, see  \cite{BCM,GM}. Generally, a hypergraph model gave a higher accuracy
%of molecular structure description \cite{KS}.

A cactus is a connected graph in which any two cycles have at most
one common vertex. Denote by $P_n$ the (ordinary) path of order $n$.
A saw-graph of order $n$ with length $k$ is a cactus of order $n$ obtained from $P_{n-k}$ by replacing $k$ of its edges with $k$ triangles, where $0\leq k\leq \lfloor\frac{n-1}{2}\rfloor$. In particular, $P_n$ is a saw-graph of length $0$.
A saw-graph  with length $k$ is a proper saw-graph if its order is $2k+1$. An end of a saw-graph is a vertex of degree $2$ that is adjacent to a vertex of degree $2$. The saw-graph obtained by joining an end of a proper saw graph of length $p$ with an end of another proper saw-graph of length $q$ by a path of length $\ell$ is denoted by  $S(p,q;\ell)$. Particularly, $S(p,q;0)$ is just the proper saw-graph of length $p+q$.

%The distance spectral radius of a connected graph is the largest eigenvalue of its distance matrix.

Let $\mathcal{C}(n,k)$ be the class of all cacti on $n$ vertices and $k$ cycles, where $0\le k\le \lfloor\frac{n-1}{2}\rfloor$. Let  $G$ be a graph with maximum distance spectral radius in $\mathcal{C}(n, k)$.  Then, by \cite[Lemma 5.2]{SMS}, all cycles of $G$ are triangles. If $G$ is not a saw-graph, then there does not necessarily exist a cut vertex $v$ such that $G-v$ has three components, for example, a graph obtained from a saw-graph on $n-2$ vertices by adding a triangle at a vertex of degree two that is not an end on some triangle. In this case,
\cite[Lemma 5.1]{SMS} does not apply. So, to show $G$ is a saw-graph, some different technique is needed.  Note that $S(0,0, n-1)$ and $S(0, 1, n-3)$ are the unique graphs with maximum distance spectral radius in
$C(n, 0)$ and $C(n, 1)$, respectively \cite{RP,SI,Yu}.  Based on further computer results,
%show that $G\cong  S(p, q;n-2k-1)$ with $p + q = k$.
Bose et al. \cite{SMS}  posed the following conjecture.

\begin{Conjecture} \cite[Conjecture 5.4]{SMS} \label{00}
$S(\lfloor\frac{k}{2}\rfloor, \lceil\frac{k}{2}\rceil;n-2k-1)$ uniquely maximizes the distance spectral radius in $\mathcal{C}(n,k)$.
\end{Conjecture}

%In this paper, we determine the unique hypertrees with maximum  distance  spectral radius among hypertrees of size at most three on $n$ vertices with $t$ edges of size three, where $1\leq t\leq \frac{n-1}{2}$, which lead to the conjecture mentioned above completely resolved.

If $T$ is a hypertree, then $O_T$ is a connected graph in which every  clique corresponds to an edge of $T$.

Let $T(n,a,b)$ be the hypertree obtained by inserting a pendant vertex
$w_i$ in  $e_i=\{v_i,v_{i+1}\}$ of the path $P_{n-a-b}=v_{1}v_{2}\cdots v_{n-a-b}$
 for $i=1, \dots, a, n-a-2b, \dots, n-a-b-1$, where $0\leq a\leq b$ and $a+b\le  \lfloor \frac{n-1}{2} \rfloor$. It is evident that $O_{T(n,a,b)}\cong S(a,b;n-2(a+b)-1)$.

In this paper, we prove the following result.

\begin{Theorem} \label{t0}
Let $T$ be a  hypertree of order $n$  with rank at most three  and  $k$ edges of size three, where $1\leq k\leq \lfloor \frac{n-1}{2} \rfloor$.
Then $\rho(T)\leq \rho (T(n,\lfloor\frac{k}{2}\rfloor, \lceil\frac{k}{2}\rceil))$ with equality if and only if $T\cong T(n,\lfloor\frac{k}{2}\rfloor, \lceil\frac{k}{2}\rceil)$.
\end{Theorem}

In terms of graphs, Theorem \ref{t0} may be rephrased as:

\begin{Theorem} \label{t1}
Suppose that  $G\in \mathcal{C}(n,k)$ and all cycles of $G$ are triangles, where $1\leq k\leq \lfloor \frac{n-1}{2} \rfloor$.
Then $\rho(G)\leq \rho (S(\lfloor\frac{k}{2}\rfloor, \lceil\frac{k}{2}\rceil;n-2k-1))$ with equality if and only if $G\cong S(\lfloor\frac{k}{2}\rfloor, \lceil\frac{k}{2}\rceil;n-2k-1)$.
\end{Theorem}

Let $G$ be a graph with maximum distance spectral radius in $\mathcal{C}(n, k)$ for $k\ge 1$, then all cycles of $G$ are triangles. So, by Theorem \ref{t1}, $G\cong S(\lfloor\frac{k}{2}\rfloor, \lceil\frac{k}{2}\rceil;n-2k-1)$. That is,  Conjecture \ref{00} is true.

\section{Preliminaries}

Let $G$ be a connected hypergraph. Since $D(G)$ is irreducible, by
Perron-Frobenius theorem, $\rho(G)$ is simple and there is a unique
unit positive eigenvector  corresponding to $\rho(G)$,
which is called the distance Perron vector of $G$, denoted by $x(G)$.

Let $V(G)=\{v_1,\dots,v_n\}$ and $x=(x_{v_1},\dots,x_{v_{n}})^{T}\in\mathbb{R}^{n}$. Then
\[
x^{\top}D(G)x=2\sum_{\{u,v\}\subseteq V(G)}d_G(u,v)x_ux_v.
 \]
If $x$ is unit and $x$ has at least one nonnegative component, then by Rayleigh's principle, we have  $\rho(G)\geq x^{\top}D(G)x$ with equality   if and only if $x=x(G)$.

For $x=x(G)$ and each $u\in V(G)$, we have
\[
\rho(G)x_u=\sum_{v\in V(G)}d_G(u,v)x_v,
 \]
which is called the distance  eigenequation of $G$ at $u$.

%\begin{Lemma}  \label{automorphism}
%\cite{LZ}  Let $G$ be a connected hypergraph with $\eta$ being an automorphism of $G$ and $x=x(G)$. Then $\eta(u)=v$ implies that $x_u=x_v$.
%\end{Lemma}

For  a connected hypergraph $G$ with  $V_1\subseteq V(G)$, let $\sigma_{G}(V_1)$ be the sum of the entries of the distance Perron vector of $G$ corresponding to the vertices in $V_1$. Furthermore, if all  the vertices of $V_1$ induce a connected subhypergraph $H$ of $G$, then we write $\sigma_{G}(H)$ instead of $\sigma_{G}(V_1)$.

For $e\in E(G)$, let $G-e$ be the subhypergraph of $G$ obtained by deleting $e$.

Here, we give a result that will be used frequently in the next section.

\begin{Lemma}  \label{66}
Let $T$ be a hypertree with two edges, say $e_1$ and $e_2$. Suppose that $u_i,v_i \in e_i$ for $i=1,2$, and $d_T(u_1, u_2)=d_T(v_1, v_2)+2$.
%$u_1$, $u_2$ lie outside the  loose path from $v_1$ to $v_2$.
For $i=1,2$, let $T_i$ be the component of  $T-e_i$ containing $u_i$ and $A_i=\{w\in V(T): d_T(w,u_i)=d_T(w,v_i)\}$. Let $x=x(T)$.
\begin{enumerate}
\item[(i)]
\begin{align*}
&\rho(T)(x_{u_1}-x_{u_2})-\rho(T)(x_{v_1}-x_{v_2})\\
=& 2(\sigma_{T}(T_2)-\sigma_{T}(T_1))+\sigma_{T}(A_2)-\sigma_{T}(A_1).
\end{align*}

\item[(ii)] If  $e_i\setminus\{u_i,v_i\}=\{w_i\}$ and $\mbox{deg}_T(w_i)=1$  for $i=1,2$, then
\[
(\rho(T)+1)(x_{w_1}-x_{w_2})-\rho(T)(x_{v_1}-x_{v_2})=x_{w_2}-x_{w_1}+\sigma_{T}(T_2)-\sigma_{T}(T_1)
\]
and
\[
\rho(T)(x_{u_1}-x_{u_2})-(\rho(T)+1)(x_{w_1}-x_{w_2})=\sigma_{T}(T_2)-\sigma_{T}(T_1).
\]
\end{enumerate}
\end{Lemma}

\begin{proof}
Let $T'_i$ be the component of  $T-e_i$ containing $v_i$ for $i=1,2$. Evidently, $V(T'_1)=(V(T'_1)\cap V(T'_2))\cup A_2 \cup V(T_2)$, and  $V(T'_1)\cap V(T'_2),  A_2,  V(T_2)$ are disjoint.
From the distance  eigenequations of $T$ at $u_1$ and $v_1$, we have
\begin{align*}
\rho(T)(x_{u_1}-x_{v_1})=&\sum_{w\in V(T)}(d_T(u_1,w)-d_T(v_1,w))x_w\\
=&\sigma_{T}(T'_1)-\sigma_{T}(T_1)\\
=&\sigma_{T}(V(T'_1)\cap V(T'_2))+\sigma_{T}(A_2)+\sigma_{T}(T_2)-\sigma_{T}(T_1)
\end{align*}
as $d_T(u_1,w)-d_T(v_1,w)=1$ if $w\in V(T'_1)$, $-1$ if $w\in V(T_1)$, and $0$ otherwise.
Similarly,
\begin{align*}
\rho(T)(x_{v_2}-x_{u_2})=&\sigma_{T}(T_2)-\sigma_{T}(T'_2)\\
=&\sigma_{T}(T_2)-\sigma_{T}(V(T'_1)\cap V(T'_2))-\sigma_{T}(A_1)-\sigma_{T}(T_1).
\end{align*}
So
\[
\rho(T)(x_{u_1}-x_{v_1})+\rho(T)(x_{v_2}-x_{u_2})=2(\sigma_{T}(T_2)-\sigma_{T}(T_1))+\sigma_{T}(A_2)-\sigma_{T}(A_1),
\]
from which Item (i) follows.

Now suppose that $\mbox{deg}_T(w_i)=1$ and $e_i\setminus\{u_i,v_i\}=\{w_i\}$ for $i=1,2$. Then  the component of  $T-e_i$ containing $w_i$ has exactly one vertex $w_i$ and $\{w\in V(T): d_T(w,w_i)=d_T(w,v_i)\}=V(T_i)$ for $i=1,2$.
Item (i) reduces to
\[
\rho(T)(x_{w_1}-x_{w_2})-\rho(T)(x_{v_1}-x_{v_2})=2(x_{w_2}-x_{w_1})+\sigma_{T}(T_2)-\sigma_{T}(T_1),
\]
from which the first equation in Item (ii) follows.

From  the distance eigenequations of $T$ at $u_1$,  $w_1$, $w_2$ and $u_2$, we have
\[
\rho(T)(x_{u_1}-x_{w_1})=x_{w_1}-\sigma_{T}(T_1)
\]
and
\[
\rho(T)(x_{w_2}-x_{u_2})=\sigma_{T}(T_2)-x_{w_2}.
\]
So
$\rho(T)(x_{u_1}-x_{u_2})-\rho(T)(x_{w_1}-x_{w_2})=x_{w_1}-x_{w_2}+\sigma_{T}(T_2)-\sigma_{T}(T_1)$, from which the second equation in Item (ii) follows.
\end{proof}

\section{Distance spectral properties of $T(n,a,b)$}

In this section, we give some properties related to the entries of the Perron vector of $T(n,a,b)$, which will be used in subsequent proof.

\begin{Lemma}  \label{ab1}
 Let $T=T(n,a,b)$, where $a\geq 0$, $b\geq a+2$ and $2(a+b)<n-1$. Let $\ell=n-a-b$. Let $x=x(T)$.
 Let $e$ be the edge containing both $v_{\ell-b}$ and $v_{\ell-b+1}$.
 Let $T_1$ and $T_2$ be the components of  $T-e$ containing $v_{\ell-b}$
 and $v_{\ell-b+1}$, respectively.
 If $b\geq \frac{\ell}{2}$.
 then $\sigma_{T}(T_1)<\sigma_{T}(T_2)$.
\end{Lemma}

\begin{proof}
Let $\ell-b=q$. Then  $\sigma_{T}(T_1)=\sum_{j=1}^{q}x_{v_{j}}+\sum_{j=1}^{a}x_{w_{j}}$ and $\sigma_{T}(T_2)=\sum_{j=q+1}^{\ell}x_{v_{j}}+\sum_{j=q+1}^{\ell-1}x_{w_{j}}$.  As $b\geq \frac{\ell}{2}$, we have $2q\leq \ell$.
Suppose that $\sigma_{T}(T_1)\geq \sigma_{T}(T_2)$.
%Evidently,
%$q-a-1\ge q-1-(b-2)=q-b+1=\ell+1$.

\noindent
{\bf Claim 1.}
$x_{v_{q-i}}\leq x_{v_{q+1+i}}$  for $0\leq i\leq q-a-1$.

We show this
by induction on $i$.
For $i=0$, from the distance eigenequations of $T$ at $v_q$ and $v_{q+1}$, we have
\[
\rho(T)(x_{v_q}-x_{v_{q+1}})=\sigma_{T}(T_2)-\sigma_{T}(T_1),
\]
so $x_{v_q}\leq x_{v_{q+1}}$.
Suppose that $1\leq i\leq q-a-1$ and $x_{v_{q-j}}\leq x_{v_{q+1+j}}$ for $0\leq j\leq i-1$.
Note that $\sigma_{T}(T_2)-\sigma_{T}(T_1)\le 0$, $\sum_{j=0}^{i-1}\left(x_{v_{q-j}}-x_{v_{q+1+j}}\right)\le 0$ and
$2\sum_{j=q+1}^{q+i}x_{w_{j}}-x_{w_{q+i}}>0$.
So, by Lemma \ref{66}(i), we have
\begin{align*}
&\rho(T)\left(x_{v_{q-i}}-x_{v_{q+1+i}}\right)-\rho(T)\left(x_{v_{q-(i-1)}}-x_{v_{q+1+(i-1)}}\right)\\
=&2\left(\sum_{j=q+1+i}^{\ell}x_{v_{j}}+\sum_{j=q+1+i}^{\ell-1}x_{w_{j}}\right)-2\left(\sum_{j=1}^{q-i}x_{v_{j}}+\sum_{j=1}^{a}x_{w_{j}}\right)+x_{w_{q+i}}\\
=&2\left(\sigma_{T}(T_2)-\sum_{j=0}^{i-1}x_{v_{q+1+j}}-\sum_{j=q+1}^{q+i}x_{w_{j}}\right)-2\left(\sigma_{T}(T_1)-\sum_{j=0}^{i-1}x_{v_{q-j}}\right)+x_{w_{q+i}}\\
=&2\left(\sigma_{T}(T_2)-\sigma_{T}(T_1)\right)+2\sum_{j=0}^{i-1}\left(x_{v_{q-j}}-x_{v_{q+1+j}}\right)
-\left(2\sum_{j=q+1}^{q+i}x_{w_{j}}-x_{w_{q+i}}\right)\\
<& 0,
\end{align*}
implying that $x_{v_{q-i}}-x_{v_{q+1+i}}< x_{v_{q-(i-1)}}-x_{v_{q+1+(i-1)}}\le 0$.
So $x_{v_{q-i}}< x_{v_{q+1+i}}$. This proves Claim 1.

\noindent
{\bf Claim 2.} $x_{v_{q-i}}\leq x_{v_{q+1+i}}$ and $x_{w_{q-i}}\leq x_{w_{q+i}}$ for $q-a\leq i\leq q-1$ with $a\geq 1$.

We show this by induction on $i$.
By Claim 1, $\sum_{j=0}^{q-a-1}\left(x_{v_{q-j}}-x_{v_{q+1+j}}\right)\leq 0$.
Then, by Lemma \ref{66}(ii), we have
\begin{align*}
&(\rho(T)+1)\left(x_{w_{a}}-x_{w_{2q-a}}\right)-\rho(T)\left(x_{v_{a+1}}-x_{v_{2q-a}}\right)\\
=&\sum_{j=2q+1-a}^{\ell}x_{v_{j}}+\sum_{j=2q-a}^{\ell-1}x_{w_{j}}-\left(\sum_{j=1}^{a}x_{v_{j}}+\sum_{j=1}^{a}x_{w_{j}}\right)\\
=&\sigma_{T}(T_2)-\sum_{j=0}^{q-a-1}x_{v_{q+1+j}}-\sum_{j=1}^{q-a-1}x_{w_{q+j}}-\left(\sigma_{T}(T_1)-\sum_{j=0}^{q-a-1}x_{v_{q-j}}\right)\\
=&(\sigma_{T}(T_2)-\sigma_{T}(T_1))+\sum_{j=0}^{q-a-1}\left(x_{v_{q-j}}-x_{v_{q+1+j}}\right)-\sum_{j=1}^{q-a-1}x_{w_{q+j}}\\
<& 0,
\end{align*}
so $(\rho(T)+1)(x_{w_{a}}-x_{w_{2q-a}})<\rho(T)(x_{v_{a+1}}-x_{v_{2q-a}})<0$, and thus $x_{w_{a}}<x_{w_{2q-a}}$.
On the other hand, by Lemma \ref{66}(ii) again, we have
\begin{align*}
&\rho(T)\left(x_{v_{a}}-x_{v_{2q+1-a}}\right)-(\rho(T)+1)\left(x_{w_{a}}-x_{w_{2q-a}}\right)\\
=&\sum_{j=2q+1-a}^{\ell}x_{v_{j}}+\sum_{j=2q+1-a}^{\ell-1}x_{w_{j}}-\left(\sum_{j=1}^{a}x_{v_{j}}+\sum_{j=1}^{a-1}x_{w_{j}}\right)\\
=&\sigma_{T}(T_2)-\sum_{j=0}^{q-a-1}x_{v_{q+1+j}}-\sum_{j=1}^{q-a}x_{w_{q+j}}-\left(\sigma_{T}(T_1)-\sum_{j=0}^{q-a-1}x_{v_{q-j}}-x_{w_a}\right)\\
=&(\sigma_{T}(T_2)-\sigma_{T}(T_1))+\sum_{j=0}^{q-a-1}\left(x_{v_{q-j}}-x_{v_{q+1+j}}\right)-\sum_{j=1}^{q-a-1}x_{w_{q+j}}+(x_{w_{a}}-x_{w_{2q-a}})\\
<& 0,
\end{align*}
so $\rho(T)(x_{v_{a}}-x_{v_{2q+1-a}})<(\rho(T)+1)(x_{w_{a}}-x_{w_{2q-a}})<0$, and thus
 $x_{v_{a}}<x_{v_{2q+1-a}}$. So Claim 2 is true for $i=q-a$.

Suppose that $q-a+1 \leq i\leq q-1$ with $a\geq 2$,  $x_{v_{q-j}}\leq x_{v_{q+1+j}}$ and $x_{w_{q-j}}\leq x_{w_{q+j}}$ for $q-a\leq j\leq i-1$. By Lemma \ref{66}(ii), we have
\begin{align*}
&(\rho(T)+1)\left(x_{w_{q-i}}-x_{w_{q+i}}\right)-\rho(T)\left(x_{v_{q-(i-1)}}-x_{v_{q+1+(i-1)}}\right)\\
=&\sum_{j=q+1+i}^{\ell}x_{v_{j}}+\sum_{j=q+i}^{\ell-1}x_{w_{j}}-\left(\sum_{j=1}^{q-i}x_{v_{j}}+\sum_{j=1}^{q-i}x_{w_{j}}\right)\\
=&\sigma_{T}(T_2)-\sum_{j=0}^{i-1}x_{v_{q+1+j}}-\sum_{j=1}^{i-1}x_{w_{q+j}}-\left(\sigma_{T}(T_1)-\sum_{j=0}^{i-1}x_{v_{q-j}}-\sum_{j=q-a}^{i-1}x_{w_{q-j}}\right)\\
=&(\sigma_{T}(T_2)-\sigma_{T}(T_1))
+\sum_{j=0}^{i-1}\left(x_{v_{q-j}}-x_{v_{q+1+j}}\right)
+\sum_{j=q-a}^{i-1}\left(x_{w_{q-j}}-x_{w_{q+j}}\right)-\sum_{j=1}^{q-a-1}x_{w_{q+j}}\\
<& 0,
\end{align*}
so $(\rho(T)+1)(x_{w_{q-i}}-x_{w_{q+i}})<\rho(T)(x_{v_{q-(i-1)}}-x_{v_{q+1+(i-1)}})\le 0$.
Thus $x_{w_{q-i}}<x_{w_{q+i}}$.
On the other hand, by Lemma \ref{66}(ii) again, we have
\begin{align*}
&\rho(T)\left(x_{v_{q-i}}-x_{v_{q+1+i}}\right)-(\rho(T)+1)\left(x_{w_{q-i}}-x_{w_{q+i}}\right)\\
=&\sum_{j=q+1+i}^{\ell}x_{v_{j}}+\sum_{j=q+1+i}^{\ell-1}x_{w_{j}}-\left(\sum_{j=1}^{q-i}x_{v_{j}}+\sum_{j=1}^{q-i-1}x_{w_{j}}\right)\\
=&\sigma_{T}(T_2)-\sum_{j=0}^{i-1}x_{v_{q+1+j}}-\sum_{j=1}^{i}x_{w_{q+j}}-\left(\sigma_{T}(T_1)-\sum_{j=0}^{i-1}x_{v_{q-j}}-\sum_{j=q-a}^{i}x_{w_{p-j}}\right)\\
=&(\sigma_{T}(T_2)-\sigma_{T}(T_1))+\sum_{j=0}^{i-1}\left(x_{v_{q-j}}-x_{v_{q+1+j}}\right)\\
&-\sum_{j=1}^{q-a-1}x_{w_{q+j}}+\sum_{j=q-a}^{i}\left(x_{w_{q-j}}-x_{w_{q+j}}\right)\\
<& 0,
\end{align*}
so $\rho(T)(x_{v_{q-i}}-x_{v_{q+1+i}})<(\rho(T)+1)(x_{w_{q-i}}-x_{w_{q+i}})<0$.
Thus $x_{v_{q-i}}<x_{v_{q+1+i}}$. This proves Claim 2.

Combining Claims 1 and 2, we conclude  that $x_{v_{q-i}}\leq x_{v_{q+1+i}}$  for $0\leq i\leq q-1$, and $x_{w_{q-i}}\leq x_{w_{q+i}}$ for $q-a\leq i\leq q-1$ with $a\geq 1$.
However, by letting $J=0$ if $b=\frac{\ell}{2}$ and  $J=\sum_{i=q}^{\ell-q-1}x_{v_{q+1+i}}+\sum_{i=q}^{\ell-q-1}x_{w_{q+i}}$ otherwise, we have
\begin{align*}
0 \geq & \sigma_{T}(T_2)-\sigma_{T}(T_1)\\
=&\sum_{i=0}^{\ell-q-1}x_{v_{q+1+i}}+\sum_{i=1}^{q-a-1}x_{w_{q+i}}+\sum_{i=q-a}^{\ell-q-1}x_{w_{q+i}}-\left(\sum_{i=0}^{q-1}x_{v_{q-i}}+\sum_{i=q-a}^{q-1}x_{w_{q-i}}\right)\\
=& \sum_{i=0}^{q-1}(x_{v_{q+1+i}}-x_{v_{q-i}})+\sum_{i=q-a}^{q-1}(x_{w_{q+i}}-x_{w_{q-i}})+\sum_{i=1}^{q-a-1}x_{w_{q+i}}+J\\
>& 0,
\end{align*}
a contradiction.  It thus follows that $\sigma_{T}(T_1)<\sigma_{T}(T_2)$.
\end{proof}

\begin{Lemma}  \label{ab-1}
 Let $T=T(n,a,b)$, where $a\geq 0$, $b\geq a+2$ and $2(a+b)<n-1$. Let $\ell=n-a-b$. Let $p=\lfloor \frac{\ell}{2}\rfloor$ and $p_1=\lceil \frac{\ell}{2}\rceil$. Let $x=x(T)$. Suppose that $b< \frac{\ell}{2}$.  Then

 (i) $x_{v_p}>x_{v_{p_1+1}};$

 (ii) $x_{v_{i}}> x_{v_{\ell+1-i}}$ and $x_{w_{i}}> x_{w_{\ell-i}}$ for $i=1,\dots,a$ with $a\ge 1$, and $x_{v_{a+1}}> x_{v_{\ell-a}}$.
 %$1\leq i\leq a$
 %$x_{v_1}>x_{v_{\ell}}$.
 %$x_{v_{i}}> x_{v_{\ell+1-i}}$ for $1\leq i\leq a+1$, $x_{w_{i}}> x_{w_{\ell-i}}$ for $1\leq i\leq a$, and $x_{w_{i-1}}-x_{w_{\ell-(i-1)}}<x_{v_{i}}-x_{v_{\ell+1-i}}$ for $2\leq i\leq a+1$.
\end{Lemma}

\begin{proof}
Note that $b\le p$.
We prove the Item (i) by considering two cases.

\noindent \textbf{Case 1.} $\ell$ is odd and $b=p$, i.e., $b=\frac{\ell-1}{2}$.

Let $T_1$ be the component of $T-\{v_b, v_{b+1}\}$ containing $v_{b}$, and
$T_2$  the component of $T-\{v_{b+1},w_{b+1},v_{b+2} \}$ containing $v_{b+2}$.
Note that $\sigma_{T}(T_1)=\sum_{j=1}^{b}x_{v_{j}}+\sum_{j=1}^{a}x_{w_{j}}$ and $\sigma_{T}(T_2)=\sum_{j=b+2}^{2b+1}x_{v_{j}}+\sum_{j=b+2}^{2b}x_{w_{j}}$.

From the distance eigenequations of $T$ at $v_b$ and $v_{b+2}$, we have
\begin{equation}\label{cc}
\rho(T)(x_{v_b}-x_{v_{b+2}})=2(\sigma_{T}(T_2)-\sigma_{T}(T_1))+x_{w_{b+1}}.
\end{equation}

Suppose that $2(\sigma_{T}(T_2)-\sigma_{T}(T_1))+x_{w_{b+1}}\leq 0$.
 Then $\sigma_{T}(T_2)< \sigma_{T}(T_1)$.

\noindent
{\bf Claim i.} $x_{v_{b-i}}\leq x_{v_{b+2+i}}$  for $0\leq i\leq b-a-1$.

We show this
by induction on $i$.
For $i=0$, from (\ref{cc}),
we have $x_{v_b}\leq x_{v_{b+2}}$.
Suppose that $1\leq i\leq b-a-1$ and $x_{v_{b-j}}\leq x_{v_{b+2+j}}$ for $0\leq j\leq i-1$.
Then,
we have by Lemma \ref{66}(i) that
\begin{align*}
&\rho(T)\left(x_{v_{b-i}}-x_{v_{b+2+i}}\right)-\rho(T)\left(x_{v_{b-(i-1)}}-x_{v_{b+2+(i-1)}}\right)\\
=&2\left(\sum_{j=b+2+i}^{2b+1}x_{v_{j}}+\sum_{j=b+2+i}^{2b}x_{w_{j}}\right)-2\left(\sum_{j=1}^{b-i}x_{v_{j}}+\sum_{j=1}^{a}x_{w_{j}}\right)+x_{w_{b+1+i}}\\
=&2\left(\sigma_{T}(T_2)-\sum_{j=0}^{i-1}x_{v_{b+2+j}}-\sum_{j=b+2}^{b+1+i}x_{w_{j}}\right)-2\left(\sigma_{T}(T_1)-\sum_{j=0}^{i-1}x_{v_{b-j}}\right)+x_{w_{b+1+i}}\\
=&2(\sigma_{T}(T_2)-\sigma_{T}(T_1))+2\sum_{j=0}^{i-1}\left(x_{v_{b-j}}-x_{v_{b+2+j}}\right)
-\left(2\sum_{j=b+2}^{b+1+i}x_{w_{j}}-x_{w_{b+1+i}}\right)\\
<& 0,
\end{align*}
implying that $x_{v_{b-i}}-x_{v_{b+2+i}}< x_{v_{b-(i-1)}}-x_{v_{b+2+(i-1)}}\le 0$, so $x_{v_{b-i}}< x_{v_{b+2+i}}$. Claim i follows.

\noindent
{\bf Claim ii.} $x_{v_{b-i}}\leq x_{v_{b+2+i}}$ and $x_{w_{b-i}}\leq x_{w_{b+1+i}}$ for $b-a\leq i\leq b-1$ with $a\geq 1$.

We show this by induction on $i$. By Claim i, $\sum_{j=0}^{b-a-1}\left(x_{v_{b-j}}-x_{v_{b+2+j}}\right)\leq 0$.
By Lemma \ref{66}(ii), we have
\begin{align*}
&(\rho(T)+1)\left(x_{w_{a}}-x_{w_{2b+1-a}}\right)-\rho(T)\left(x_{v_{a+1}}-x_{v_{2b-a+1}}\right)\\
=&\sum_{j=2b+2-a}^{2b+1}x_{v_{j}}+\sum_{j=2b+1-a}^{2b}x_{w_{j}}-\left(\sum_{j=1}^{a}x_{v_{j}}+\sum_{j=1}^{a}x_{w_{j}}\right)\\
=&\sigma_{T}(T_2)-\sum_{j=0}^{b-a-1}x_{v_{b+2+j}}-\sum_{j=1}^{b-a-1}x_{w_{b+1+j}}-\left(\sigma_{T}(T_1)-\sum_{j=0}^{b-a-1}x_{v_{b-j}}\right)\\
=&(\sigma_{T}(T_2)-\sigma_{T}(T_1))+\sum_{j=0}^{b-a-1}\left(x_{v_{b-j}}-x_{v_{b+2+j}}\right)-\sum_{j=1}^{b-a-1}x_{w_{b+1+j}}\\
<& 0,
\end{align*}
implying that $(\rho(T)+1)(x_{w_{a}}-x_{w_{2b+1-a}})<\rho(T)(x_{v_{a+1}}-x_{v_{2b-a+1}})<0$, so
 $x_{w_{a}}<x_{w_{2b+1-a}}$.
On the other hand, by Lemma \ref{66}(ii) again, we have
\begin{align*}
&\rho(T)\left(x_{v_{a}}-x_{v_{2b+2-a}}\right)-(\rho(T)+1)\left(x_{w_{a}}-x_{w_{2b+1-a}}\right)\\
=&\sum_{j=2b+2-a}^{2b+1}x_{v_{j}}+\sum_{j=2b+2-a}^{2b}x_{w_{j}}-\left(\sum_{j=1}^{a}x_{v_{j}}+\sum_{j=1}^{a-1}x_{w_{j}}\right)\\
=&\sigma_{T}(T_2)-\sum_{j=0}^{b-a-1}x_{v_{b+2+j}}-\sum_{j=1}^{b-a}x_{w_{b+1+j}}-\left(\sigma_{T}(T_1)-\sum_{j=0}^{b-a-1}x_{v_{b-j}}-x_{w_a}\right)\\
=&(\sigma_{T}(T_2)-\sigma_{T}(T_1))
+\sum_{j=0}^{b-a-1}\left(x_{v_{b-j}}-x_{v_{b+2+j}}\right)\\
&+(x_{w_{a}}-x_{w_{2b+1-a}})-\sum_{j=1}^{b-a-1}x_{w_{b+1+j}}\\
<& 0,
\end{align*}
so $\rho(T)(x_{v_{a}}-x_{v_{2b+2-a}})<(\rho(T)+1)(x_{w_{a}}-x_{w_{2b+1-a}})<0$, and
 thus $x_{v_{a}}<x_{v_{2b+2-a}}$. So Claim ii is true for $i=b-a$.

Suppose that  $b-a+1\leq i \leq b-1$ with $a\geq 2$,   $x_{v_{b-j}}\leq x_{v_{b+2+j}}$ and $x_{w_{b-j}}\leq x_{w_{b+1+j}}$ for $b-a\leq j\leq i-1$.   By Lemma \ref{66}(ii), we have
\begin{align*}
&(\rho(T)+1)\left(x_{w_{b-i}}-x_{w_{b+1+i}}\right)-\rho(T)\left(x_{v_{b-(i-1)}}-x_{v_{b+2+(i-1)}}\right)\\
=&\sum_{j=b+2+i}^{2b+1}x_{v_{j}}+\sum_{j=b+1+i}^{2b}x_{w_{j}}-\left(\sum_{j=1}^{b-i}x_{v_{j}}+\sum_{j=1}^{b-i}x_{w_{j}}\right)\\
=&\sigma_{T}(T_2)-\sum_{j=0}^{i-1}x_{v_{b+2+j}}-\sum_{j=1}^{i-1}x_{w_{b+1+j}}-\left(\sigma_{T}(T_1)-\sum_{j=0}^{i-1}x_{v_{b-j}}-\sum_{j=b-a}^{i-1}x_{w_{b-j}}\right)\\
=&(\sigma_{T}(T_2)-\sigma_{T}(T_1))+\sum_{j=0}^{i-1}\left(x_{v_{b-j}}-x_{v_{b+2+j}}\right)\\
&+\sum_{j=b-a}^{i-1}\left(x_{w_{b-j}}-x_{w_{b+1+j}}\right)-\sum_{j=1}^{b-a-1}x_{w_{b+1+j}}\\
<& 0,
\end{align*}
so $(\rho(T)+1)(x_{w_{b-i}}-x_{w_{b+1+i}})<\rho(T)(x_{v_{b-(i-1)}}-x_{v_{b+2+(i-1)}})\le 0$.
Thus $x_{w_{b-i}}<x_{w_{b+1+i}}$.
On the other hand, by Lemma \ref{66}(ii) again, we have
\begin{align*}
&\rho(T)\left(x_{v_{b-i}}-x_{v_{b+2+i}}\right)-(\rho(T)+1)\left(x_{w_{b-i}}-x_{w_{b+1+i}}\right)\\
=&\sum_{j=b+2+i}^{2b+1}x_{v_{j}}+\sum_{j=b+2+i}^{2b}x_{w_{j}}-\left(\sum_{j=1}^{b-i}x_{v_{j}}+\sum_{j=1}^{b-i-1}x_{w_{j}}\right)\\
=&\sigma_{T}(T_2)-\sum_{j=0}^{i-1}x_{v_{b+2+j}}-\sum_{j=1}^{i}x_{w_{b+1+j}}-\left(\sigma_{T}(T_1)-\sum_{j=0}^{i-1}x_{v_{b-j}}-\sum_{j=b-a}^{i}x_{w_{p-j}}\right)\\
=&(\sigma_{T}(T_2)-\sigma_{T}(T_1))+\sum_{j=0}^{i-1}\left(x_{v_{b-j}}-x_{v_{b+2+j}}\right)\\
&-\sum_{j=1}^{b-a-1}x_{w_{b+1+j}}+\sum_{j=b-a}^{i}\left(x_{w_{b-j}}-x_{w_{b+1+j}}\right)\\
<& 0,
\end{align*}
so $\rho(T)(x_{v_{b-i}}-x_{v_{b+2+i}})<(\rho(T)+1)(x_{w_{b-i}}-x_{w_{b+1+i}})<0$, and
thus $x_{v_{b-i}}<x_{v_{b+2+i}}$. This proves Claim ii.

Combining Claims i and ii, we have $x_{v_{b-i}}\leq x_{v_{b+2+i}}$  for $0\leq i\leq b-1$, and $x_{w_{b-i}}\leq x_{w_{b+1+i}}$ for $b-a\leq i\leq b-1$ with $a\geq 1$.
However,
\begin{align*}
0 \geq &\sigma_{T}(T_2)-\sigma_{T}(T_1)\\
=&\sum_{i=0}^{b-1}x_{v_{b+2+i}}+\sum_{i=1}^{b-a-1}x_{w_{b+1+i}}+\sum_{i=b-a}^{b-1}x_{w_{b+1+i}}-\left(\sum_{i=0}^{b-1}x_{v_{b-i}}+\sum_{i=b-a}^{b-1}x_{w_{b-i}}\right)\\
=& \sum_{i=0}^{b-1}(x_{v_{b+2+i}}-x_{v_{b-i}})+\sum_{i=b-a}^{b-1}(x_{w_{b+1+i}}-x_{w_{b-i}})+\sum_{i=1}^{b-a-1}x_{w_{b+1+i}}\\
>& 0,
\end{align*}
which is a contradiction.  Therefore, $2(\sigma_{T}(T_2)-\sigma_{T}(T_1))+x_{w_{b+1}}>0$, and from (\ref{cc}), we have $x_{v_b}>x_{v_{b+2}}$, i.e., $x_{v_p}>x_{v_{p_1+1}}$.

\noindent \textbf{Case 2.} $b<p$.

Let $e=\{v_p,v_{p_1+1}\}$ if $\ell$ is even, and $e=\{v_p,v_{p_1}\}$ otherwise. As $b<p$, we have $e\in E(T)$.
Let $T_1$ be the component of $T-e$ containing $v_p$, and
$T_2$  the component of $T-v_{p_1}v_{p_1+1}$ containing $v_{p_1+1}$.
Note that $\sigma_{T}(T_1)=\sum_{j=1}^{p}x_{v_{j}}+\sum_{j=1}^{a}x_{w_{j}}$ and $\sigma_{T}(T_2)=\sum_{j=p_1+1}^{\ell}x_{v_{j}}+\sum_{j=l-b}^{\ell-1}x_{w_{j}}$.
In the following, we show that $\sigma_{T}(T_2)-\sigma_{T}(T_1)>0$. Suppose that this is not true, i.e., $\sigma_{T}(T_2)-\sigma_{T}(T_1) \leq 0$.

\noindent
{\bf Claim a.}  $x_{v_{p-i}}\leq x_{v_{p_1+1+i}}$  for $0\leq i\leq p-a-1$.

We show this
by induction on $i$.
For $i=0$, from the distance eigenequations of $T$ at $v_p$ and $v_{p_1+1}$, we get
\begin{equation}\label{c1}
\rho(T)(x_{v_p}-x_{v_{p_1+1}})=(p_1+1-p)(\sigma_{T}(T_2)-\sigma_{T}(T_1)).
\end{equation}
%from (\ref{c1}),
so $x_{v_p}\leq x_{v_{p_1+1}}$.
Suppose that $1\leq i\leq p-a-1$ and $x_{v_{p-j}}\leq x_{v_{p_1+1+j}}$ for $0\leq j\leq i-1$.
We consider $i\le p-b-1$ with $p-b\geq 2$, and $i\ge p-b$ separately.
In the former case,
we have by Lemma \ref{66}(i) that
\begin{align*}
&\rho(T)\left(x_{v_{p-i}}-x_{v_{p_1+1+i}}\right)-\rho(T)\left(x_{v_{p-(i-1)}}-x_{v_{p_1+1+(i-1)}}\right)\\
=&2\left(\sum_{j=p_1+1+i}^{\ell}x_{v_{j}}+\sum_{j=\ell-b}^{\ell-1}x_{w_{j}}\right)-2\left(\sum_{j=1}^{p-i}x_{v_{j}}+\sum_{j=\ell}^{a}x_{w_{j}}\right)\\
=&2\left(\sigma_{T}(T_2)-\sum_{j=0}^{i-1}x_{v_{p_1+1+j}}\right)-2\left(\sigma_{T}(T_1)-\sum_{j=0}^{i-1}x_{v_{p-j}}\right)\\
=&2\left(\sigma_{T}(T_2)-\sigma_{T}(T_1)\right)+2\sum_{j=0}^{i-1}\left(x_{v_{p-j}}-x_{v_{p_1+1+j}}\right)\\
\leq& 0,
\end{align*}
so $x_{v_{p-i}}-x_{v_{p_1+1+i}}\leq x_{v_{p-(i-1)}}-x_{v_{p_1+1+(i-1)}}\le 0$,
and thus $x_{v_{p-i}}\leq x_{v_{p_1+1+i}}$. In the latter case,
we have by Lemma \ref{66}(i) that
\begin{align*}
&\rho(T)\left(x_{v_{p-i}}-x_{v_{p_1+1+i}}\right)-\rho(T)\left(x_{v_{p-(i-1)}}-x_{v_{p_1+1+(i-1)}}\right)\\
=&2\left(\sum_{j=p_1+1+i}^{\ell}x_{v_{j}}+\sum_{j=p_1+1+i}^{\ell-1}x_{w_{j}}\right)-2\left(\sum_{j=1}^{p-i}x_{v_{j}}+\sum_{j=1}^{a}x_{w_{j}}\right)+x_{w_{p_1+i}}\\
=&2\left(\sigma_{T}(T_2)-\sum_{j=0}^{i-1}x_{v_{p_1+1+j}}-\sum_{j=\ell-b}^{p_1+i}x_{w_{j}}\right)-2\left(\sigma_{T}(T_1)-\sum_{j=0}^{i-1}x_{v_{p-j}}\right)+x_{w_{p_1+i}}\\
=&2(\sigma_{T}(T_2)-\sigma_{T}(T_1))+2\sum_{j=0}^{i-1}\left(x_{v_{p-j}}-x_{v_{p_1+1+j}}\right)-2\sum_{j=\ell-b}^{p_1+i}x_{w_{j}}+x_{w_{p_1+i}}\\
<& 0,
\end{align*}
so $x_{v_{p-i}}-x_{v_{p_1+1+i}}< x_{v_{p-(i-1)}}-x_{v_{p_1+1+(i-1)}}\le 0$,
and thus $x_{v_{p-i}}< x_{v_{p_1+1+i}}$. This proves Claim a.

\noindent
{\bf Claim b.} $x_{v_{p-i}}\leq x_{v_{p_1+1+i}}$
and $x_{w_{p-i}}\leq x_{w_{p_1+i}}$ for $p-a\leq i\leq p-1$ with $a\geq 1$

We show this by induction on $i$. By Claim a, $\sum_{j=0}^{p-a-1}\left(x_{v_{p-j}}-x_{v_{p_1+1+j}}\right)\leq 0$.

By Lemma \ref{66}(ii), we have
\begin{align*}
&(\rho(T)+1)\left(x_{w_{a}}-x_{w_{l-a}}\right)-\rho(T)\left(x_{v_{a+1}}-x_{v_{l-a}}\right)\\
=&\sum_{j=\ell+1-a}^{\ell}x_{v_{j}}+\sum_{j=\ell-a}^{\ell-1}x_{w_{j}}-\left(\sum_{j=1}^{a}x_{v_{j}}+\sum_{j=1}^{a}x_{w_{j}}\right)\\
=&\sigma_{T}(T_2)-\sum_{j=0}^{p-a-1}x_{v_{p_1+1+j}}-\sum_{j=p-b}^{p-a-1}x_{w_{p_1+j}}-\left(\sigma_{T}(T_1)-\sum_{j=0}^{p-a-1}x_{v_{p-j}}\right)\\
=&(\sigma_{T}(T_2)-\sigma_{T}(T_1))+\sum_{j=0}^{p-a-1}\left(x_{v_{p-j}}-x_{v_{p_1+1+j}}\right)-\sum_{j=p-b}^{p-a-1}x_{w_{p_1+i}}\\
<& 0,
\end{align*}
so $(\rho(T)+1)(x_{w_{a}}-x_{w_{\ell-a}})<\rho(T)(x_{v_{a+1}}-x_{v_{\ell-a}})<0$, and
thus  $x_{w_{a}}<x_{w_{\ell-a}}$.
On the other hand, by Lemma \ref{66}(ii) again, we have
\begin{align*}
&\rho(T)\left(x_{v_{a}}-x_{v_{\ell+1-a}}\right)-(\rho(T)+1)\left(x_{w_{a}}-x_{w_{\ell-a}}\right)\\
=&\sum_{j=\ell+1-a}^{\ell}x_{v_{j}}+\sum_{j=\ell+1-a}^{\ell-1}x_{w_{j}}-\left(\sum_{j=1}^{a}x_{v_{j}}+\sum_{j=1}^{a-1}x_{w_{j}}\right)\\
=&\sigma_{T}(T_2)-\sum_{j=0}^{p-a-1}x_{v_{p_1+1+j}}-\sum_{j=p-b}^{p-a}x_{w_{p_1+j}}-\left(\sigma_{T}(T_1)-\sum_{j=0}^{p-a-1}x_{v_{p-j}}-x_{w_a}\right)\\
=&(\sigma_{T}(T_2)-\sigma_{T}(T_1))+\sum_{j=0}^{p-a-1}\left(x_{v_{p-j}}-x_{v_{p_1+1+j}}\right)\\
&-\sum_{j=p-b}^{p-a-1}x_{w_{p_1+j}}+(x_{w_{a}}-x_{w_{l-a}})\\
<& 0,
\end{align*}
so $\rho(T)(x_{v_{a}}-x_{v_{\ell+1-a}})<(\rho(T)+1)(x_{w_{a}}-x_{w_{\ell-a}})<0$, and
thus $x_{v_{a}}<x_{v_{\ell+1-a}}$. So Claim b is true for $i=p-a$.

Suppose that  $p-a+1\leq i\leq p-1$ with $a\geq 2$,  $x_{v_{p-j}}\leq x_{v_{p_1+1+j}}$ and  $x_{w_{p-j}}\leq x_{w_{p_1+j}}$ for $p-a\leq j\leq i-1$.
By Lemma \ref{66}(ii), we have
\begin{align*}
&(\rho(T)+1)\left(x_{w_{p-i}}-x_{w_{p_1+i}}\right)-\rho(T)\left(x_{v_{p-(i-1)}}-x_{v_{p_1+1+(i-1)}}\right)\\
=&\sum_{j=p_1+1+i}^{\ell}x_{v_{j}}+\sum_{j=p_1+i}^{\ell-1}x_{w_{j}}-\left(\sum_{j=1}^{p-i}x_{v_{j}}+\sum_{j=1}^{p-i}x_{w_{j}}\right)\\
=&\sigma_{T}(T_2)-\sum_{j=0}^{i-1}x_{v_{p_1+1+j}}-\sum_{j=p-b}^{i-1}x_{w_{p_1+j}}-\left(\sigma_{T}(T_1)-\sum_{j=0}^{i-1}x_{v_{p-j}}-\sum_{j=p-a}^{i-1}x_{w_{p-j}}\right)\\
=&(\sigma_{T}(T_2)-\sigma_{T}(T_1))+\sum_{j=0}^{i-1}\left(x_{v_{p-j}}-x_{v_{p_1+1+j}}\right)\\
&-\sum_{j=p-b}^{p-a-1}x_{w_{p_1+j}}+\sum_{j=p-a}^{i-1}\left(x_{w_{p-j}}-x_{w_{p_1+j}}\right)\\
<& 0,
\end{align*}
so $(\rho(T)+1)(x_{w_{p-i}}-x_{w_{p_1+i}})<\rho(T)(x_{v_{p-(i-1)}}-x_{v_{p_1+1+(i-1)}})\le 0$,
and thus $x_{w_{p-i}}<x_{w_{p_1+i}}$.
On the other hand, by Lemma \ref{66}(ii) again, we have
\begin{align*}
&\rho(T)\left(x_{v_{p-i}}-x_{v_{p_1+1+i}}\right)-(\rho(T)+1)\left(x_{w_{p-i}}-x_{w_{p_1+i}}\right)\\
=&\sum_{j=p_1+1+i}^{\ell}x_{v_{j}}+\sum_{j=p_1+1+i}^{\ell-1}x_{w_{j}}-\left(\sum_{j=1}^{p-i}x_{v_{j}}+\sum_{j=1}^{p-i-1}x_{w_{j}}\right)\\
=&\sigma_{T}(T_2)-\sum_{j=0}^{i-1}x_{v_{p_1+1+j}}-\sum_{j=p-b}^{i}x_{w_{p_1+j}}-\left(\sigma_{T}(T_1)-\sum_{j=0}^{i-1}x_{v_{p-j}}-\sum_{j=p-a}^{i}x_{w_{p-j}}\right)\\
=&(\sigma_{T}(T_2)-\sigma_{T}(T_1))+\sum_{j=0}^{i-1}\left(x_{v_{p-j}}-x_{v_{p_1+1+j}}\right)\\
&-\sum_{j=p-b}^{p-a-1}x_{w_{p_1+j}}+\sum_{j=p-a}^{i}\left(x_{w_{p-j}}-x_{w_{p_1+j}}\right)\\
<& 0,
\end{align*}
so $\rho(T)(x_{v_{p-i}}-x_{v_{p_1+1+i}})<(\rho(T)+1)(x_{w_{p-i}}-x_{w_{p_1+i}})<0$, and
thus $x_{v_{p-i}}<x_{v_{p_1+1+i}}$. This proves Claim b.

Combining Claims a and b,  one has $x_{v_{p-i}}\leq x_{v_{p_1+1+i}}$  for $0\leq i\leq p-1$, and $x_{w_{p-i}}\leq x_{v_{p_1+i}}$ for $p-a\leq i\leq p-1$ with $a\geq 1$.
However,
\begin{align*}
0 \geq &\sigma_{T}(T_2)-\sigma_{T}(T_1)\\
=&\sum_{i=0}^{p-1}x_{v_{p_1+1+i}}+\sum_{i=p-b}^{p-a-1}x_{w_{p_1+i}}+\sum_{i=p-a}^{p-1}x_{w_{p_1+i}}-\left(\sum_{i=0}^{p-1}x_{v_{p-i}}+\sum_{i=p-a}^{p-1}x_{w_{p-i}}\right)\\
=& \sum_{i=0}^{p-1}(x_{v_{p_1+1+i}}-x_{v_{p-i}})+\sum_{i=p-a}^{p-1}(x_{w_{p_1+i}}-x_{w_{p-i}})+\sum_{i=p-b}^{p-a-1}x_{w_{p_1+i}}\\
>& 0,
\end{align*}
a contradiction.  Therefore, $\sigma_{T}(T_2)-\sigma_{T}(T_1)>0$, and from (\ref{c1}), we have $x_{v_p}>x_{v_{p_1+1}}$.

By combining the above two cases, we have $x_{v_p}>x_{v_{p_1+1}}$, so Item (i) follows.

In the following, we prove Item (ii). Recall  that $a\ge 1$.

Suppose that $x_{v_1}\le x_{v_{\ell}}$.

First we show that $x_{v_i}\le x_{v_{\ell+1-i}}$  and $x_{w_{i}}\leq x_{w_{\ell-i}}$ for $1\le i\le a$ by induction $i$.
For $i=1$, by Lemma \ref{66}(ii), we have
%\[
%\rho(T)(x_{v_1}-x_{v_\ell})-(\rho(T)+1)(x_{w_1}-x_{w_{\ell-1}})=x_{v_\ell}-x_{v_1},
%\]
%i.e.,
\[
(\rho(T)+1)(x_{v_1}-x_{v_\ell})=(\rho(T)+1)(x_{w_1}-x_{w_{\ell-1}}).
\]
As $x_{v_1}\le x_{v_{\ell}}$,  so $x_{w_1}\leq x_{w_{\ell-1}}$.
%We show that $x_{v_i}\le x_{v_{\ell+1-i}}$ for $1\le i\le p$ and $x_{w_{i}}\leq x_{w_{\ell-i}}$ for $1\le i\le a$ by induction $i$.
Suppose that $2\leq i\leq a$ with $a\geq 2$,   $x_{v_{j}}\leq x_{v_{\ell+1-j}}$  and $x_{w_{j}}\leq x_{w_{\ell-j}}$ for $1\leq j\leq i-1$.
By Lemma \ref{66}(ii), we have
\begin{align*}
&(\rho(T)+1)\left(x_{w_{i-1}}-x_{w_{\ell-(i-1)}}\right)-\rho(T)\left(x_{v_{i}}-x_{v_{l+1-i}}\right)\\
=&\sum_{j=\ell-(i-1)+1}^{\ell}x_{v_{j}}+\sum_{j=\ell-(i-1)}^{\ell-1}x_{w_{j}}-\left(\sum_{j=1}^{i-1}x_{v_{j}}+\sum_{j=1}^{i-1}x_{w_{j}}\right)\\
=&\sum_{j=1}^{i-1}\left(x_{v_{\ell+1-j}}-x_{v_{j}}\right)+\sum_{j=1}^{i-1}\left(x_{w_{\ell-j}}-x_{w_{j}}\right)\\
\geq& 0,
\end{align*}
so  $\rho(T)(x_{v_{i}}-x_{v_{\ell+1-i}})\leq (\rho(T)+1)(x_{w_{i-1}}-x_{w_{\ell-(i-1)}})\leq 0$, and thus $x_{v_{i}}\leq x_{v_{\ell+1-i}}$.
On the other hand, by Lemma \ref{66}(ii) again, we have
\begin{align*}
&\rho(T)\left(x_{v_{i}}-x_{v_{\ell+1-i}}\right)-(\rho(T)+1)\left(x_{w_{i}}-x_{w_{\ell-i}}\right)\\
=&\sum_{j=\ell+1+i}^{\ell}x_{v_{j}}+\sum_{j=\ell+1+i}^{\ell-1}x_{w_{j}}-\left(\sum_{j=1}^{i}x_{v_{j}}+\sum_{j=1}^{i-1}x_{w_{j}}\right)\\
=&\sum_{j=1}^{i}\left(x_{v_{\ell+1-j}}-x_{v_{j}}\right)+\sum_{j=1}^{i-1}\left(x_{w_{\ell-j}}-x_{w_{j}}\right)\\
\geq& 0,
\end{align*}
so $(\rho(T)+1)(x_{w_{i}}-x_{w_{\ell-i}})\leq\rho(T)(x_{v_{i}}-x_{v_{\ell+1-i}})\le 0$, and
thus $x_{w_{i}}\leq x_{w_{\ell-i}}$. Therefore,
$x_{v_{i}}\leq x_{v_{\ell+1-i}}$ and $x_{w_{i}}\leq x_{w_{\ell-i}}$
if $1\le i\le a$.

Next we show that $x_{v_i}\le x_{v_{\ell+1-i}}$ for $a+1\leq i\leq p$ by induction on $i$.
By above proof, $\sum_{j=1}^{a}\left(x_{v_{\ell+1-j}}-x_{v_{j}}\right)\geq 0$ and $\sum_{j=1}^{a}\left(x_{w_{\ell-j}}-x_{w_{j}}\right)\geq 0$.
For $i=a+1$, by Lemma \ref{66}(ii), we have
\begin{align*}
&(\rho(T)+1)\left(x_{w_{a}}-x_{w_{\ell-a}}\right)-\rho(T)\left(x_{v_{a+1}}-x_{v_{\ell-a}}\right)\\
=&\sum_{j=\ell-a+1}^{\ell}x_{v_{j}}+\sum_{j=\ell-a}^{\ell-1}x_{w_{j}}-\left(\sum_{j=1}^{a}x_{v_{j}}+\sum_{j=1}^{a}x_{w_{j}}\right)\\
=&\sum_{j=1}^{a}\left(x_{v_{\ell+1-j}}-x_{v_{j}}\right)+\sum_{j=1}^{a}\left(x_{w_{\ell-j}}-x_{w_{j}}\right)\\
\geq& 0,
\end{align*}
so $\rho(T)(x_{v_{a+1}}-x_{v_{\ell-a}})\leq (\rho(T)+1)(x_{w_{a}}-x_{w_{\ell-a}})\le 0$, and
thus $x_{v_{a+1}}\leq x_{v_{\ell-a}}$.

Suppose that $a+2\leq i\leq p$ and $x_{v_{j}}\leq x_{v_{\ell+1-j}}$ for $a+1\leq j\leq i-1$.
Let $E=0$ if $i=a+2$, and $E=2\sum_{j=\ell+1-(i-1)}^{\ell-a-1}x_{w_{j}}$  otherwise. Let $F=x_{w_{\ell+1-i}}$ if $i\leq b+1$ with $b<p$, or $i\leq b$ with $b=p$, and $F=0$ otherwise. Evidently, $E\geq 0$ and $F\geq 0$.
By Lemma \ref{66}(i), we have
\begin{align*}
&\rho(T)\left(x_{v_{i-1}}-x_{v_{\ell+1-(i-1)}}\right)-\rho(T)\left(x_{v_{i}}-x_{v_{\ell+1-i}}\right)\\
=&2\left(\sum_{j=\ell+1-(i-1)}^{\ell}x_{v_{j}}+\sum_{j=\ell+1-(i-1)}^{\ell-1}x_{w_{j}}\right)-2\left(\sum_{j=1}^{i-1}x_{v_{j}}+\sum_{j=1}^{a}x_{w_{j}}\right)+F\\
=&2\sum_{j=1}^{i-1}\left(x_{v_{\ell+1-j}}-x_{v_{j}}\right)+2\sum_{j=1}^{a}\left(x_{w_{\ell-j}}-x_{w_{j}}\right)+E+F\\
\geq & 0,
\end{align*}
so $x_{v_{i}}-x_{v_{\ell+1-i}}\leq x_{v_{i-1}}-x_{v_{\ell+1-(i-1)}}\le 0$,
and thus $x_{v_{i}}\leq x_{v_{\ell+1-i}}$.

Now it follows that $x_{v_{i}}\leq x_{v_{\ell+1-i}}$ for  $1\leq i\leq p$.
 In particular, we have $x_{v_p}\leq x_{v_{p_1+1}}$, contradicting to Item (i). Therefore, we have
  $x_{v_1}> x_{v_\ell}$. So, by similar arguments as above by induction on $i$, we  have $x_{v_{i}}> x_{v_{\ell+1-i}}$ and $x_{w_{i}}> x_{w_{\ell-i}}$ for $i=1,\dots,a$, and $x_{v_{a+1}}> x_{v_{\ell-a}}$. This is Item (ii).
\end{proof}

\begin{Lemma}  \label{ab}
 Let $T=T(n,a,b)$, where $a\geq 0$, $b\geq a+2$ and $2(a+b)<n-1$. Let $\ell=n-a-b$. Let $x=x(T)$. If $b< \frac{\ell}{2}$.  Then

(i)
 $x_{v_{i}}-x_{v_{\ell+1-i}}<x_{v_{i+1}}-x_{v_{\ell+1-(i+1)}}$ and $ x_{w_{i}}-x_{w_{\ell-i}}<x_{v_{i+1}}-x_{v_{\ell-i}}$ for  $i=1, \dots, a$ with $a\ge 1$;

%(i)
% $0< x_{v_{i}}-x_{v_{\ell+1-i}}<x_{v_{i+1}}-x_{v_{\ell+1-(i+1)}}$ and $0< x_{w_{i}}-x_{w_{\ell-i}}<x_{v_{i+1}}-x_{v_{\ell-i}}$ for  $i=1, \dots, a$;

(ii) $x_{v_{a+1+i}}-x_{v_{\ell-b+1-i}}>x_{v_{a+1+(i+1)}}-x_{v_{\ell-b+1-(i+1)}}>0$ for
$ i=1, \dots, \lfloor\frac{\ell-b-a-1}{2}\rfloor-1$;

(iii) $x_{v_{\ell-b+1}}<x_{v_{\ell-a}}$.
\end{Lemma}

\begin{proof}
By Lemma \ref{66}(i) and Lemma \ref{ab-1}(ii), we have
\begin{align*}
\rho(T)\left(x_{v_{1}}-x_{v_{\ell}}\right)-\rho(T)\left(x_{v_{2}}-x_{v_{\ell-1}}\right)
=&2(x_{v_{\ell}}-x_{v_{1}})+(x_{w_{\ell-1}}-x_{w_1})\\
< & 0
\end{align*}
and
\begin{align*}
&\rho(T)\left(x_{v_{i}}-x_{v_{\ell+1-i}}\right)-\rho(T)\left(x_{v_{i+1}}-x_{v_{\ell+1-(i+1)}}\right)\\
=&2\left(\sum_{j=\ell+1-i}^{\ell}x_{v_{j}}+\sum_{j=\ell+1-i}^{\ell-1}x_{w_{j}}\right)
-2\left(\sum_{j=1}^{i}x_{v_{j}}+\sum_{j=1}^{i-1}x_{w_{j}}\right)+x_{w_{\ell-i}}-x_{w_{i}}\\
=&2\sum_{j=1}^{i}\left(x_{v_{\ell+1-j}}-x_{v_{j}}\right)
+2\sum_{j=1}^{i-1}\left(x_{w_{\ell-j}}-x_{w_{j}}\right)+x_{w_{\ell-i}}-x_{w_{i}}\\
< & 0
\end{align*}
for  $i=2,\dots, a$.
So $x_{v_{i}}-x_{v_{\ell+1-i}}< x_{v_{i+1}}-x_{v_{\ell+1-(i+1)}}$ for $i=1, \dots, a$. This proves the first part of Item (i).

For  $i=1, \dots, a$, by Lemma \ref{66}(ii) and Lemma \ref{ab-1}(ii), we have
\begin{align*}
&(\rho(T)+1)\left(x_{w_{i}}-x_{w_{\ell-i}}\right)-\rho(T)\left(x_{v_{i+1}}-x_{v_{\ell+1-(i+1)}}\right)\\
=&\sum_{j=\ell-i+1}^{\ell}x_{v_{j}}+\sum_{j=\ell-i}^{\ell-1}x_{w_{j}}-\left(\sum_{j=1}^{i}x_{v_{j}}+\sum_{j=1}^{i}x_{w_{j}}\right)\\
=&\sum_{j=1}^{i}\left(x_{v_{\ell+1-j}}-x_{v_{j}}\right)+\sum_{j=1}^{i}\left(x_{w_{\ell-j}}-x_{w_{j}}\right)\\
<& 0,
\end{align*}
so
\[
0<(\rho(T)+1)(x_{w_{i}}-x_{w_{\ell-i}})<\rho(T)(x_{v_{i+1}}-x_{v_{\ell+1-(i+1)}}),
\]
and thus $x_{w_{i}}-x_{w_{\ell-i}}<x_{v_{i+1}}-x_{v_{\ell-i}}$ for  $i=1, \dots, a$. This proves the second part of Item (i).

Next we prove Item (ii).

Let $t=\lfloor \frac{\ell-b+a+1}{2}\rfloor$ and $t_1=\lceil \frac{\ell-b+a+1}{2}\rceil$.
It suffices to prove that $x_{v_{t-i}}-x_{t_1+1+i}>x_{v_{t-(i-1)}}-x_{v_{t_1+1+(i-1)}}>0$ for
$0\leq i \leq t-(a+2)$. Let  $S_1=\sum_{j=1}^{t}x_{v_{j}}+\sum_{j=1}^{a}x_{w_{j}}$ and $S_2=\sum_{j=t_1+1}^{\ell}x_{v_{j}}+\sum_{j=\ell-b}^{\ell-1}x_{w_{j}}$.
First we show that $S_1<S_2$.

As in the proof of Lemma \ref{ab-1}, let $p=\lfloor \frac{\ell}{2}\rfloor$ and $p_1=\lceil \frac{\ell}{2}\rceil$. Recall that $b\leq p$.

If $\ell$ is odd and $b=p$ (i.e., $b=\frac{\ell-1}{2}$), then, by the arguments in  Case 1 of the proof of Item (i) in Lemma \ref{ab-1}, we have
\[
\sum_{j=1}^{b}x_{v_{j}}+\sum_{j=1}^{a}x_{w_{j}}< \sum_{j=b+2}^{2b+1}x_{v_{j}}+\sum_{j=b+2}^{2b}x_{w_{j}}+\frac{1}{2}x_{w_{b+1}}
<\sum_{j=b+1}^{2b+1}x_{v_{j}}+\sum_{j=b+1}^{2b}x_{w_{j}}.
\]
As $t_1\leq b$, we have $S_1<S_2$.
Suppose  that  $b<p$. Then,  by the arguments in Case 2 of the proof of Item (i) in Lemma \ref{ab-1}, we have
\[\sum_{j=1}^{p}x_{v_{j}}+\sum_{j=1}^{a}x_{w_{j}}<
\sum_{j=p_1+1}^{\ell}x_{v_{j}}+\sum_{j=\ell-b}^{\ell-1}x_{w_{j}}.
\]
As
$t\leq p$ and $t_1\leq p_1$, we have $S_1<S_2$.
It follows that $S_1<S_2$ in either case.

Next we prove that $x_{v_{t-i}}> x_{v_{t_1+1+i}}$  for $0\leq i\leq t-(a+2)$ by induction on $i$.

For $i=0$, from the distance eigenequations of $T$ at $v_t$ and $v_{t_1+1}$, we have
$\rho(T)(x_{v_t}- x_{v_{t_1+1}})=(t_1+1-t)(S_2-S_1)>0$, so $x_{v_t}>x_{v_{t_1+1}}$.

Suppose that $1\leq i\leq t-(a+2)$, and $x_{v_{t-j}}> x_{v_{t_1+1+j}}$ for $0\leq j\leq i-1$.
By Lemma \ref{66}(i), we have
\begin{align*}
&\rho(T)\left(x_{v_{t-i}}-x_{v_{t_1+1+i}}\right)-\rho(T)\left(x_{v_{t-(i-1)}}-x_{v_{t_1+1+(i-1)}}\right)\\
=&2\left(\sum_{j=t_1+1+i}^{\ell}x_{v_{j}}+\sum_{j=\ell-b}^{\ell-1}x_{w_{j}}\right)-2\left(\sum_{j=1}^{t-i}x_{v_{j}}+\sum_{j=1}^{a}x_{w_{j}}\right)\\
=&2\left(S_2-\sum_{j=0}^{i-1}x_{v_{t_1+1+j}}\right)-2\left(S_1-\sum_{j=0}^{i-1}x_{v_{t-j}}\right)\\
=&2\left(S_2-S_1\right)+2\sum_{j=0}^{i-1}\left(x_{v_{t-j}}-x_{v_{t_1+1+j}}\right)\\
>& 0,
\end{align*}
so $x_{v_{t-i}}-x_{v_{t_1+1+i}}> x_{v_{t-(i-1)}}-x_{v_{t_1+1+(i-1)}}>0$ for $1\leq i\leq t-(a+2)$. This proves Item (ii).

In the following, we  prove Item (iii).
By Lemma \ref{ab-1}(ii), we have $x_{v_i}>x_{v_{\ell+1-i}}$ and $x_{w_i}>x_{w_{\ell-i}}$ for $1\leq i\leq a$ with $a\ge 1$.
Thus
\[
\sum_{i=1}^{a}(x_{v_{i}}+x_{w_{i}})>\sum_{i=\ell-a+1}^{\ell}x_{v_{i}}+ \sum_{i=\ell-a}^{\ell-1}x_{w_{i}}.
\]
Let $m=\ell-\lfloor \frac{a+b}{2}\rfloor$. We consider two cases.

\noindent \textbf{Case 1.} $a+b$ is even, i.e., $b-a$ is even.

Let $R_1=\sum_{j=1}^{m}x_{v_{j}}+\sum_{j=1}^{a}x_{w_{j}}+\sum_{j=l-b}^{m-1}x_{w_{j}}$ and $R_2=\sum_{j=m+1}^{\ell}x_{v_{j}}+\sum_{j=m+1}^{\ell-1}x_{w_{j}}$.
We claim that $x_{v_{m-i}}-x_{v_{m+1-i}}$ and $R_2-R_1$ have common sign for $0\leq i\leq \frac{b-a}{2}-1$, and $x_{w_{m-i}}-x_{w_{m
+i}}$ and $R_2-R_1$ have common sign for $1\leq i\leq \frac{b-a}{2}-1$.

For $i=0$, from the distance eigenequations of $T$ at $v_m$ and $v_{m+1}$, we have
$\rho(T)(x_{v_m}-x_{v_{m+1}})=R_2-R_1$, so $x_{v_m}-x_{v_{m+1}}$ and $R_2-R_1$ have common sign.

For $i=1$ with $b-a\geq 4$,  by Lemma \ref{66}(ii), we have
\begin{align*}
&(\rho(T)+1)\left(x_{w_{m-1}}-x_{w_{m+1}}\right)-\rho(T)\left(x_{v_{m}}-x_{v_{m+1}}\right)\\
=&\sum_{j=m+2}^{\ell}x_{v_{j}}+\sum_{j=m+1}^{\ell-1}x_{w_{j}}
-\left(\sum_{j=1}^{m-1}x_{v_{j}}+\sum_{j=1}^{a}x_{w_{j}}+\sum_{j=\ell-b}^{m-1}x_{w_{j}}\right)\\
=&\left(R_2-x_{v_{m+1}}\right)-\left(R_1-x_{v_m}\right),
\end{align*}
so $(\rho(T)+1)(x_{w_{m-1}}-x_{w_{m+1}})=(R_2-R_1)+(x_{v_m}-x_{v_{m+1}})+\rho(T)\left(x_{v_{m}}-x_{v_{m+1}}\right)$,  and
thus $x_{w_{m-1}}-x_{w_{m+1}}$ and $R_2-R_1$ have common sign.

 By Lemma \ref{66}(ii) again, we have
\begin{align*}
&\rho(T)\left(x_{v_{m-1}}-x_{v_{m+2}}\right)-(\rho(T)+1)\left(x_{w_{m-1}}-x_{w_{m+1}}\right)\\
=&\sum_{j=m+2}^{\ell}x_{v_{j}}+\sum_{j=m+2}^{\ell-1}x_{w_{j}}-\left(\sum_{j=1}^{m-1}x_{v_{j}}+\sum_{j=1}^{a}x_{w_{j}}+\sum_{j=\ell-b}^{m-2}x_{w_{j}}\right)\\
=&\left(R_2-x_{v_{m+1}}-x_{w_{m+1}}\right)-\left(R_1-x_{v_m}-x_{w_{m-1}}\right),
\end{align*}
so
\begin{align*}
\rho(T)\left(x_{v_{m-1}}-x_{v_{m+2}}\right)=&(R_2-R_1)+(x_{v_m}-x_{v_{m+1}})
+(x_{w_{m-1}}-x_{w_{m+1}})\\
&+(\rho(T)+1)\left(x_{w_{m-1}}-x_{w_{m+1}}\right),
\end{align*}
and
thus $x_{v_{m-1}}-x_{v_{m+2}}$ and $R_2-R_1$ have common sign.

Suppose that $2\leq i \leq \frac{b-a}{2}-1$ with $b-a\geq 6$, and $x_{v_{m-j}}-x_{v_{m+1+j}}$ and $R_2-R_1$ have common sign for $0\leq j\leq i-1$, and $x_{w_{m-j}}-x_{w_{m+j}}$ and $R_2-R_1$ have common sign for $1\leq j\leq i-1$.
By Lemma \ref{66}(ii), we have
\begin{align*}
&(\rho(T)+1)\left(x_{w_{m-i}}-x_{w_{m+i}}\right)-\rho(T)\left(x_{v_{m-(i-1)}}-x_{v_{m+1+(i-1)}}\right)\\
=&\sum_{j=m+1+i}^{\ell}x_{v_{j}}+\sum_{j=m+i}^{\ell-1}x_{w_{j}}
-\left(\sum_{j=1}^{m-i}x_{v_{j}}+\sum_{j=1}^{a}x_{w_{j}}+\sum_{j=\ell-b}^{m-i}x_{w_{j}}\right)\\
=&\left(R_2-\sum_{j=0}^{i-1}x_{v_{m+1+j}}-\sum_{j=1}^{i-1}x_{w_{m+j}}\right)
-\left(R_1-\sum_{j=0}^{i-1}x_{v_{m-j}}-\sum_{j=1}^{i-1}x_{w_{m-j}}\right),
\end{align*}
so
\begin{align*}
(\rho(T)+1)\left(x_{w_{m-i}}-x_{w_{m+i}}\right)=&(R_2-R_1)+\sum_{j=0}^{i-1}\left(x_{v_{m-j}}
-x_{v_{m+1+j}}\right)\\
&+\sum_{j=1}^{i-1}(x_{w_{m-j}}-x_{w_{m+j}})\\
&+\rho(T)\left(x_{v_{m-(i-1)}}-x_{v_{m+1+(i-1)}}\right),
\end{align*}
and
thus $x_{w_{m-i}}-x_{w_{m+i}}$ and $R_2-R_1$ have common sign.

By Lemma \ref{66}(ii) again, we have
\begin{align*}
&\rho(T)\left(x_{v_{m-i}}-x_{v_{m+1+i}}\right)-(\rho(T)+1)\left(x_{w_{m-i}}-x_{w_{m+i}}\right)\\
=&\sum_{j=m+1+i}^{\ell}x_{v_{j}}+\sum_{j=m+1+i}^{\ell-1}x_{w_{j}}-\left(\sum_{j=1}^{m-i}x_{v_{j}}+\sum_{j=1}^{a}x_{w_{j}}+\sum_{j=\ell-b}^{m-i-1}x_{w_{j}}\right)\\
=&\left(R_2-\sum_{j=0}^{i-1}x_{v_{m+1+j}}-\sum_{j=1}^{i}x_{w_{m+j}}\right)-\left(R_1-\sum_{j=0}^{i-1}x_{v_{m-j}}-\sum_{j=1}^{i}x_{w_{m-j}}\right),
\end{align*}
so
\begin{align*}
\rho(T)\left(x_{v_{m-i}}-x_{v_{m+1+i}}\right)=&(R_2-R_1)
+\sum_{j=0}^{i-1}\left(x_{v_{m-j}}-x_{v_{m+1+j}}\right)\\
&+\sum_{j=1}^{i}(x_{w_{m-j}}-x_{w_{m+j}})+(\rho(T)+1)\left(x_{w_{m-i}}-x_{w_{m+i}}\right),
\end{align*}
and
thus $x_{v_{m-i}}-x_{v_{m+1+i}}$ and $R_2-R_1$ have common sign.

Now we have showed  that $x_{v_{m-i}}-x_{v_{m+1-i}}$ and $R_2-R_1$ have common sign for $0\leq i\leq \frac{b-a}{2}-1$, and $x_{w_{m-i}}-x_{w_{m
+i}}$ and $R_2-R_1$ have common sign for $1\leq i\leq \frac{b-a}{2}-1$, as claimed.
Note that
\begin{align*}
R_2-R_1<&\left(\sum_{i=\ell-a+1}^{\ell}x_{v_{i}}+ \sum_{i=\ell-a}^{\ell-1}x_{w_{i}}\right)+\left(\sum_{i=0}^{\frac{b-a}{2}-1}x_{v_{m+1+i}}+ \sum_{i=1}^{\frac{b-a}{2}-1}x_{w_{m+i}}\right) \\
&-\sum_{i=1}^{a}\left(x_{v_{i}}+x_{w_{i}}\right)-\left(\sum_{i=0}^{\frac{b-a}{2}-1}x_{v_{m-i}}+ \sum_{i=1}^{\frac{b-a}{2}-1}x_{w_{m-i}}\right)\\
<& -\left(\sum_{i=0}^{\frac{b-a}{2}-1}(x_{v_{m-i}}-x_{v_{m+1-i}})+\sum_{i=1}^{\frac{b-a}{2}-1}(x_{w_{m-i}}-x_{w_{m+i}})\right).
\end{align*}
This requires the above common sign to be $-$, and thus $x_{v_{m-\left(\frac{b-a}{2}-1\right)}}<x_{v_{m+1+\left(\frac{b-a}{2}-1\right)}}$, i.e., $x_{v_{\ell-b+1}}<x_{v_{\ell-a}}$, as desired.

\noindent \textbf{Case 2.} $a+b$ is odd, i.e., $b-a$ is odd.

Let $B_1=\sum_{j=1}^{m-1}x_{v_{j}}+\sum_{j=1}^{a}x_{w_{j}}+\sum_{j=\ell-b}^{m-1}x_{w_{j}}$ and $B_2=\sum_{j=m+1}^{\ell}x_{v_{j}}+\sum_{j=m}^{\ell-1}x_{w_{j}}$.
We claim that $x_{w_{m-i}}-x_{w_{m-1+i}}$  and $B_2-B_1$ have common sign, and $x_{v_{m-i}}-x_{v_{m+i}}$ and $B_2-B_1$ have common sign for $1\leq i\leq \frac{b-a-1}{2}$.

For $i=1$, from the distance eigenequations of $T$ at $w_{m-1}$ and $w_{m}$, we have
$(\rho(T)+1)(x_{w_{m-1}}-x_{w_{m}})=B_2-B_1$, so $x_{w_{m-1}}-x_{w_{m}}$ and $B_2-B_1$ have common sign.
By Lemma \ref{66}(ii), we have
\begin{align*}
&\rho(T)\left(x_{v_{m-1}}-x_{v_{m+1}}\right)-(\rho(T)+1)\left(x_{w_{m-1}}-x_{w_{m}}\right)\\
=&\sum_{j=m+1}^{\ell}x_{v_{j}}+\sum_{j=m+1}^{\ell-1}x_{w_{j}}-\left(\sum_{j=1}^{m-1}x_{v_{j}}+\sum_{j=1}^{a}x_{w_{j}}+\sum_{j=\ell-b}^{m-2}x_{w_{j}}\right)\\
=&\left(B_2-x_{w_{m}}\right)-\left(B_1-x_{w_{m-1}}\right),
\end{align*}
so
\begin{align*}
\rho(T)\left(x_{v_{m-1}}-x_{v_{m+1}}\right)=&(B_2-B_1)+(x_{w_{m-1}}-x_{w_{m}})\\
&+(\rho(T)+1)\left(x_{w_{m-1}}-x_{w_{m}}\right),
\end{align*}
and thus $x_{v_{m-1}}-x_{v_{m+1}}$
and $B_2-B_1$ have common sign.

Suppose that $2\leq i \leq \frac{b-a-1}{2}$ with $b-a\geq 5$, and $x_{v_{m-j}}-x_{v_{m+j}}$ and $B_2-B_1$ have common sign, and $x_{w_{m-j}}-x_{w_{m-1+j}}$ and $B_2-B_1$ have common sign for $1\leq j\leq i-1$.
By Lemma \ref{66}(ii), we have
\begin{align*}
&(\rho(T)+1)\left(x_{w_{m-i}}-x_{w_{m-1+i}}\right)-\rho(T)\left(x_{v_{m-(i-1)}}-x_{v_{m+(i-1)}}\right)\\
=&\sum_{j=m+i}^{\ell}x_{v_{j}}+\sum_{j=m-1+i}^{\ell-1}x_{w_{j}}-\left(\sum_{j=1}^{m-i}x_{v_{j}}+\sum_{j=1}^{a}x_{w_{j}}+\sum_{j=\ell-b}^{m-i}x_{w_{j}}\right)\\
=&\left(B_2-\sum_{j=1}^{i-1}x_{v_{m+j}}-\sum_{j=1}^{i-1}x_{w_{m-1+j}}\right)-\left(B_1-\sum_{j=1}^{i-1}x_{v_{m-j}}-\sum_{j=1}^{i-1}x_{w_{m-j}}\right),
\end{align*}
so
\begin{align*}
(\rho(T)+1)\left(x_{w_{m-i}}-x_{w_{m-1+i}}\right)=&
(B_2-B_1)+\sum_{j=1}^{i-1}\left(x_{v_{m-j}}-x_{v_{m+j}}\right)\\
&+\sum_{j=1}^{i-1}(x_{w_{m-j}}-x_{w_{m-1-j}})\\
&+\rho(T)\left(x_{v_{m-(i-1)}}-x_{v_{m+(i-1)}}\right),
\end{align*}
and thus $x_{w_{m-i}}-x_{w_{m-1+i}}$ and $B_2-B_1$ have common sign.

By Lemma \ref{66}(ii) again, we have
\begin{align*}
&\rho(T)\left(x_{v_{m-i}}-x_{v_{m+i}}\right)-(\rho(T)+1)\left(x_{w_{m-i}}-x_{w_{m-1+i}}\right)\\
=&\sum_{j=m+i}^{\ell}x_{v_{j}}+\sum_{j=m+i}^{\ell-1}x_{w_{j}}-\left(\sum_{j=1}^{m-i}x_{v_{j}}+\sum_{j=1}^{a}x_{w_{j}}+\sum_{j=\ell-b}^{m-i-1}x_{w_{j}}\right)\\
=&\left(B_2-\sum_{j=1}^{i-1}x_{v_{m+j}}-\sum_{j=1}^{i}x_{w_{m-1+j}}\right)-\left(B_1-\sum_{j=1}^{i-1}x_{v_{m-j}}-\sum_{j=1}^{i}x_{w_{m-j}}\right),
\end{align*}
so
\begin{align*}
\rho(T)\left(x_{v_{m-i}}-x_{v_{m+i}}\right)=&
(B_2-B_1)+\sum_{j=1}^{i-1}\left(x_{v_{m-j}}-x_{v_{m+j}}\right)\\
&+\sum_{j=1}^{i}(x_{w_{m-j}}-x_{w_{m-1+j}})\\
&+(\rho(T)+1)\left(x_{w_{m-i}}-x_{w_{m-1+i}}\right),
\end{align*}
and thus $x_{v_{m-i}}-x_{v_{m+i}}$ and $B_2-B_1$ have common sign.

Now we conclude  that $x_{v_{m-i}}-x_{v_{m+i}}$ and $B_2-B_1$ have common sign, and $x_{w_{m-i}}-x_{w_{m-1+i}}$ and $B_2-B_1$ have common sign for $1\leq i\leq \frac{b-a-1}{2}$, as claimed.
Note that
\begin{align*}
B_2-B_1<&\sum_{i=\ell-a+1}^{\ell}x_{v_{i}}+ \sum_{i=\ell-a}^{\ell-1}x_{w_{i}}+\sum_{i=1}^{\frac{b-a-1}{2}}x_{v_{m+i}}+ \sum_{i=1}^{\frac{b-a-1}{2}}x_{w_{m-1+i}}\\
&-\sum_{i=1}^{a}\left(x_{v_{i}}+ x_{w_{i}}\right)-\left(\sum_{i=1}^{\frac{b-a-1}{2}}x_{v_{m-i}}+ \sum_{i=1}^{\frac{b-a-1}{2}}x_{w_{m-i}}\right)\\
<& -\sum_{i=1}^{\frac{b-a-1}{2}}\left((x_{v_{m-i}}-x_{v_{m+i}})+(x_{w_{m-i}}-x_{w_{m-1+i}})\right).
\end{align*}
This requires the above common sign to be $-$, and thus $x_{v_{m-\frac{b-a-1}{2}}}<x_{v_{m+\frac{b-a-1}{2}}}$, i.e., $x_{v_{\ell-b+1}}<x_{v_{\ell-a}}$, as desired.
\end{proof}

Let $G$ be a connected hypergraph.
For $u\in V(G)$, the status (or transmission) of $u$ in $G$, denoted by $s_G(u)$, is defined to be the sum of distances from $u$ to all other vertices of $G$, i.e., the row sum of $D(G)$ indexed by vertex $u$, i.e., $s_G(u)=\sum_{v\in V(G)}d_{G}(u,v)$. Let $s(G)=\min\{s_G(u):u\in V(G)\}$. It is known that $\rho(G)\geq s(G)$, see ~\cite[p.~24, Theorem~1.1]{Mi}.

\begin{Lemma}  \label{sum}
 Let $T=T(n,a,b)$, where $a\geq 0$, $b\geq a+2$ and $2(a+b)<n-1$. Let $\ell=n-a-b$ and $r=\ell-b-a$. Let $x=x(T)$.
If $b< \frac{\ell}{2}$,  then $\rho(T)>(2a+1)(r-1)+\frac{r}{r-1}\sum_{i=1}^{\lfloor\frac{r-1}{2}\rfloor}(r-2i)$.
\end{Lemma}

\begin{proof}
As $b< \frac{\ell}{2}$ and $b\geq a+2$, we have  $r>2b-b-(b-2)=2$.
Note that
\[
\frac{r}{r-1}\sum_{i=1}^{\lfloor\frac{r-1}{2}\rfloor}(r-2i)
<\sum_{i=1}^{\lfloor\frac{r-1}{2}\rfloor}(r-2i)+\left\lfloor\frac{r-1}{2}\right\rfloor.
\]
This is because
\begin{align*}
&\frac{r}{r-1}\sum_{i=1}^{\lfloor\frac{r-1}{2}\rfloor}(r-2i)
-\left(\sum_{i=1}^{\lfloor\frac{r-1}{2}\rfloor}(r-2i)+\left\lfloor\frac{r-1}{2}\right\rfloor\right)\\
=&\frac{1}{r-1}\sum_{i=1}^{\lfloor\frac{r-1}{2}\rfloor}(r-2i)-\left\lfloor\frac{r-1}{2}\right\rfloor\\
=&\frac{\lfloor\frac{r-1}{2}\rfloor}{r-1}\cdot\frac{(r-2+r-2\lfloor\frac{r-1}{2}\rfloor)}{2}
-\left\lfloor\frac{r-1}{2}\right\rfloor\\
=&\left\lfloor\frac{r-1}{2}\right\rfloor\left(1-\frac{\lfloor\frac{r-1}{2}\rfloor}{r-1}-1\right)\\
=&-\left\lfloor\frac{r-1}{2}\right\rfloor \frac{\lfloor\frac{r-1}{2}\rfloor}{r-1}\\
<&0.
\end{align*}

Let $T'=T(\ell-b+3a+2,a,a+1)$. As $T'$ is a proper induced subgraph of $T$ and the distance between any two vertices in $T'$ remains unchanged, we have
$s(T)>s(T')$. Thus it is sufficient to prove that $s(T')\geq (2a+1)r+\sum_{i=1}^{\lfloor\frac{r-1}{2}\rfloor}(r-2i)+\lfloor\frac{r-1}{2}\rfloor$.
Obviously, $s_{T'}(w_i)\geq s_{T'}(v_i)$ for $1\leq i\leq a$ and $\ell-b\leq i\leq \ell-b+a$.

Let $t=\lfloor \frac{\ell-b+a+1}{2}\rfloor$ and $t_1=\lceil \frac{\ell-b+a+1}{2}\rceil$ . We claim that $s(T')=s_{T'}(v_{t+1})$.
We consider two cases.

\noindent \textbf{Case 1.} $r$ is even, i.e., $\ell-b+a$ is even.

Let $V_1=\{v_1, \dots, v_{t+1}, w_1, \dots, w_a\}$ and $V_2=V(T')\setminus V_1$. We have $|V_1|=t+1+a$ and $|V_2|=t+1+a$.
As $ s_{T'}(v_{t+1})-s_{T'}(v_{t+2})=\sum_{u\in V_1 }(d_{T'}(v_{t+1},u)-d_{T'}(v_{t+2},u))+\sum_{u\in V_2}(d_{T'}(v_{t+1},u)-d_{T'}(v_{t+2},u))=\sum_{u\in V_1 }(-1)+\sum_{u\in V_2} 1=-|V_1|+|V_2|=0$, we have $ s_{T'}(v_{t+1})=s_{T'}(v_{t+2})$.  By  similar argument as above, we have $s_{T'}(v_{i-1})>s_{T'}(v_{i})$ for $2\leq i\leq t+1$ and $s_{T'}(v_{i})<s_{T'}(v_{i+1})$ for $t+2\leq i\leq \ell-b+a+1$.
Thus $s(T')=s_{T'}(v_{t+1})$.

\noindent \textbf{Case 2.} $r$ is odd, i.e., $\ell-b+a$ is odd.

Let $V_1=\{v_1, \dots, v_{t}, w_1, \dots, w_a\}$ and $V_2=V(T')\setminus V_1$. We have $|V_1|=t+a$ and $|V_2|=t+a+1$.
As $ s_{T'}(v_{t})-s_{T'}(v_{t+1})=\sum_{u\in V_1 }(d_{T'}(v_{t},u)-d_{T'}(v_{t+1},u))+\sum_{u\in V_2}(d_{T'}(v_{t},u)-d_{T'}(v_{t+1},u))=\sum_{u\in V_1 }(-1)+\sum_{u\in V_2} 1=-|V_1|+|V_2|=1$, we have $ s_{T'}(v_{t})>s_{T'}(v_{t+1})$.  By  similar argument as above, we have $s_{T'}(v_{i-1})>s_{T'}(v_{i})$ for $2\leq i\leq t$ and $s_{T'}(v_{i})<s_{T'}(v_{i+1})$ for $t+1\leq i\leq \ell-b+a+1$.
Thus $s(T')=s_{T'}(v_{t+1})$.

Note that
\begin{align*}
\sum_{i=1}^{a+1}\left(d_{T'}(v_{t+1},v_{i})+d_{T'}(v_{t+1},v_{\ell-b+a+2-i})\right)
=&\sum_{i=1}^{a+1}(\ell-b+a+2-2i)\\
\geq & \sum_{i=1}^{a+1}r\\
=&  (a+1)r,
\end{align*}
and when $a\geq 1$,
\begin{align*}
\sum_{i=1}^{a}\left(d_{T’}(v_{t+1},w_{i})+d_{T'}(v_{t+1},w_{\ell-b+a+1-i})\right)
=&\sum_{i=1}^{a}(\ell-b+a+2-2i)\\
\geq & \sum_{i=1}^{a}(r+2)\\
>& ar
\end{align*}
and
\begin{align*}
&\sum_{i=0}^{t-(a+2)}\left(d_{T'}(v_{t+1},v_{t-i})+d_{T'}(v_{t+1},v_{t_1+1+i})\right)+d_{T'}(v_{t+1},w_{\ell-b})\\
=&\sum_{i=0}^{t-(a+2)}(t_1-t+1+2i)+(\ell-b-t)\\
=&\sum_{i=1}^{\lfloor\frac{r-1}{2}\rfloor}(r-2i)+\left\lceil\frac{r-1}{2}\right\rceil\\
\geq & \sum_{i=1}^{\lfloor\frac{r-1}{2}\rfloor}(r-2i)+\left\lfloor\frac{r-1}{2}\right\rfloor.
\end{align*}
Thus  $s_{T'}(v_{t+1})\geq (2a+1)r+\sum_{i=1}^{\lfloor\frac{r-1}{2}\rfloor}(r-2i)+\left\lfloor\frac{r-1}{2}\right\rfloor$, as desired.
\end{proof}

\section{Graft transformations that increase the distance spectral radius}

Let $G$ be a  hypergraph with $u,v\in V(G)$ and $e_1,\dots,e_r\in E(G)$ such that $u\notin e_i$ and $v\in e_i$ for $1\leq i\leq r$. Let $e'_i=(e_i\setminus \{v\})\cup \{u\} $ for $1\leq i\leq r$. Suppose that $e'_i\not\in E(G)$ for $1\leq i\leq r$. Let $G'$ be the hypergraph with $V(G')=V(G)$ and $E(G')=(E(G)\setminus \{e_1,\dots,e_r\})\cup \{e'_1,\dots,e'_r\}$. Then we say that $G'$ is obtained from $G$ by moving edges $e_1,\dots,e_r$ from $v$ to $u$.

\begin{Lemma}\cite{WZ2}   \label{3}
For $t\geq 3$, let $G$ be a hypergraph  consisting of $t$ connected subhypergraphs $G_1, \dots, G_t$  such that $|V(G_i)|\geq 2$ for $1\le i\le t$  and  $V(G_i)\cap V(G_j)=\{u\}$ for $1\le i<j\le t$.
Suppose that
$\emptyset \neq I\subseteq \{3,\dots, t\}$.  Let $v\in V(G_2)\setminus \{u\}$ and $G'$ be the hypergraph obtained from $G$ by moving  all the edges containing $u$ in $G_i$ for all $i\in I$ from $u$ to $v$. If $\sigma_{G}(G_1)\geq\sigma_{G}(G_2)$, then $\rho(G)<\rho(G')$.
\end{Lemma}

Let $G$ be a  hypergraph with $e_1,e_2\in E(G)$ and $u_1, \dots, u_s\in V(G)$ such that $u_1, \dots, u_s\notin e_1$ and $u_1, \dots, u_s\in e_2$, where $|e_2|-s\ge 2$. Let $e'_1=e_1\cup \{u_1, \dots, u_s\}$ and $e'_2=e_2\setminus\{u_1, \dots, u_s\}$. Suppose that $e'_1, e'_2\not\in E(G)$. Let $G'$ be the hypergraph with $V(G')=V(G)$ and $E(G')=(E(G)\setminus \{e_1,e_2\})\cup \{e'_1,e'_2\}$. Then we say that $G'$ is obtained from $G$ by moving vertices $u_1, \dots, u_s$ from $e_2$ to $e_1$.

\begin{Lemma} \cite{WZ2} \label{edge3}
For $t\geq 3$, let $G$ be a hypergraph with an edge $e=\{w_1,\ldots,w_t\}$, such that $G-e$  consists of vertex-disjoint connected subhypergraphs $H_1, \dots, H_t$, each containing exactly one vertex of $e$.  Let $e\cap V(H_i)=\{w_i\}$ for $i=1,\dots, t$. Suppose  that $|V(H_i)|\geq 2$ for $i=1,2$.
Let $\emptyset \neq I\subseteq \{3,\ldots,t\}$. Let $e'\in E(H_2)$  and $G'$ be the hypergraph obtained from  $G$  by moving all the vertices in $\{w_i: i\in I\}$   from $e$ to $e'$.
If $\sigma_{G}(H_1)\geq\sigma_{G}(H_2)$, then $\rho(G)<\rho(G')$.
\end{Lemma}

\begin{Lemma}  \label{entry}
Suppose that  $v,w$ be two non-adjacent neighbors of vertex $u$ in a connected hypergraph $G$. Let $x=x(G)$. Then $x_w+x_u-x_v> 0$.
\end {Lemma}

\begin{proof}
Let $V_1=V(G)\setminus \{u,v,w\}$.
For $z\in V_1$, one has
$d_{G}(w,z)\geq1$ and  $d_{G}(u,z)-d_{G}(v,z)\geq-d_G(u,v)=-1$, so $d_{G}(w,z)+d_{G}(u,z)-d_{G}(v,z)\geq0$.
From the distance eigenequations of $G$ at $w$, $u$ and $v$, we have
\[
\rho(G)x_{w}=x_{u}+2x_{v}+\sum_{z\in V_1}d_{G}(w,z)x_z,
\]
\[
\rho(G)x_{u}=x_{w}+x_{v}+\sum_{z\in V_1}d_{G}(u,z)x_z,
\]
and
\[
\rho(G)x_{v}=2x_{w}+x_{u}+\sum_{z\in V_1}d_{G}(v,z)x_z.
\]
Thus
\begin{align*}
&\rho(G)(x_{w}+x_{u}-x_{v})\\
=& -x_{w}+3x_{v}+\sum_{z\in V_1}(d_{G}(w,z)+d_{G}(u,z)-d_{G}(v,z))x_z\\
\geq& -x_{w}+3x_{v},
\end{align*}
which implies $(\rho(G)+1)(x_{w}+x_{u}-x_{v})\geq x_{u}+ 2x_{v}>0$. So it follows that $x_{w}+x_{u}-x_{v}>0$.
\end{proof}

\begin{Lemma}  \label{ab-2}
 Suppose that  $b\geq a+2$ and $2(a+b)<n-1$.  Then $\rho (T(n,a+1,b-1))>\rho (T(n,a,b))$.
\end{Lemma}

\begin{proof}
Let $T=T(n,a,b)$ and $x=x(T)$. Let $\ell=n-a-b$ and $r=\ell-b-a$.

Let $T'$ be the hypergraph obtained from $T$ by moving vertex $w_{\ell-b}$ from $e_{\ell-b}$ to $e_{a+1}$. Obviously, $T'\cong T(n,a+1,b-1)$.

\noindent \textbf{Case 1.} $b\geq \frac{\ell}{2}$.

 Let $e$ be the edge containing both $v_{\ell-b}$ and $v_{\ell-b+1}$.
 Let $T_1$ be the component of $T-e$ containing $v_{\ell-b}$.
 Let $T_2$ be the component of $T-e$ containing $v_{\ell-b+1}$. By Lemma \ref{ab1}, we have  $\sigma_{T}(T_1)<\sigma_{T}(T_2)$. Then by Lemma \ref{edge3}, we have  $\rho(T')>\rho(T)$.

\noindent \textbf{Case 2.} $b< \frac{\ell}{2}$.

Let $A=\{v_{\ell-b+1}, \dots, v_{\ell}\}\cup \{w_{\ell-b+1}, \dots, w_{\ell-1}\}$, and $B=\{v_{1}, \dots, v_{a+1} \}\cup \{w_{1}, \dots, w_a\}$.
As we pass from   $T$ to $T'$, the distance between $w_{l-b}$  and a vertex of $A$  is increased by $r-1$,
the distance between $w_{\ell-b}$  and a vertex of $B$  is decreased by $r-1$,
and  the distance between $w_{\ell-b}$  and  $v_{\ell-b+1-i}$  is increased by $r-2i$ for $ i=1, \dots, r-1$, and the distance between any other vertex pair remains unchanged.  So
\[
\frac{1}{2}(\rho(T')-\rho(T))\geq\frac{1}{2}x^{\top}(D(T')-D(T))x=x_{w_{\ell-b}}W,
\]
where
$W=(r-1)(\sigma_{T}(A)-\sigma_{T}(B))+C$,
\[
C=\sum_{i=1}^{r-1}(r-2i)x_{v_{\ell-b+1-i}}=
\sum_{i=1}^{\lfloor\frac{r-1}{2}\rfloor}(r-2i)(x_{v_{\ell-b+1-i}}-x_{v_{a+1+i}}),
\]
and
\begin{align*}
\sigma_{T}(A)-\sigma_{T}(B)=&\sum_{i=\ell-b+1}^{\ell}x_{v_i}+\sum_{i=\ell-b+1}^{\ell-1}x_{w_i}-\sum_{i=1}^{a+1}x_{v_i}-\sum_{i=1}^{a}x_{w_i}\\
=&\sum_{i=1}^{a+1}(x_{v_{\ell+1-i}}-x_{v_i})+\sum_{i=1}^{a}(x_{w_{\ell-i}}-x_{w_i})+\sum_{i=\ell-b+1}^{\ell-a-1}(x_{v_{i}}+x_{w_{i}}).
\end{align*}

Now  we prove that $x_{v_{a+1}}-x_{v_{\ell-b+1}}> x_{v_{a+2}}-x_{v_{\ell-b}}$.
Suppose to the contrary that $x_{v_{a+1}}-x_{v_{\ell-b+1}}\leq x_{v_{a+2}}-x_{v_{\ell-b}}$.
Then, by Lemma \ref{66}(i), we have
\begin{align*}
2(\sigma_{T}(A)-\sigma_{T}(B))+x_{w_{\ell-b}}=&
\rho(T)(x_{v_{a+1}}-x_{v_{\ell-b+1}})-\rho(T)(x_{v_{a+2}}-x_{v_{\ell-b}})\\
\le & 0.
\end{align*}
So $\sigma_{T}(A)-\sigma_{T}(B)<0$, and $C<0$ by Lemma \ref{ab}(ii). It follows that $W<0$.

By Lemma \ref{entry}, we have $x_{w_{\ell-b}}\leq x_{v_{\ell-b+1}}+x_{w_{\ell-b+1}}$, so $x_{w_{\ell-b}}\leq \sum_{i=\ell-b+1}^{\ell-a-1}(x_{v_{i}}+x_{w_{i}})$.
So, from the distance  eigenequations of $T$ at $v_{a+2}$ and $v_{\ell-b}$,
 we have
\begin{align*}
&\rho(T)(x_{v_{a+2}}-x_{v_{\ell-b}})\\
=&(r-2)(\sigma_{T}(A)-\sigma_{T}(B))+C+(r-2)x_{w_{\ell-b}}\\
%\leq&(r-2)(\sigma_{T}(A)-\sigma_{T}(B))+C+(r-2)\sum_{i=l-b+1}^{l-a-1}(x_{v_{i}}+x_{w_{i}})\\
<&(r-2)(\sigma_{T}(A)-\sigma_{T}(B))+C+(r-1)\sum_{i=\ell-b+1}^{\ell-a-1}(x_{v_{i}}+x_{w_{i}})\\
=& (r-2)(\sigma_{T}(A)-\sigma_{T}(B))+C+(r-1)(\sigma_{T}(A)-\sigma_{T}(B))\\
&-(r-1)\left(\sum_{i=1}^{a+1}(x_{v_{\ell+1-i}}-x_{v_i})+\sum_{i=1}^{a}(x_{w_{\ell-i}}-x_{w_i})\right)\\
=&\frac{r-2}{r-1}W-\frac{r-2}{r-1}C+W\\
&+(r-1)\left(\sum_{i=1}^{a+1}(x_{v_i}-x_{v_{\ell+1-i}})+\sum_{i=1}^{a}(x_{w_i}-x_{w_{\ell-i}})\right)\\
=&\frac{2r-3}{r-1}W+\frac{r-2}{r-1}(-C)\\
&+(r-1)\left(\sum_{i=1}^{a+1}(x_{v_i}-x_{v_{\ell+1-i}})+\sum_{i=1}^{a}(x_{w_i}-x_{w_{\ell-i}})\right)\\
<&\frac{2r-3}{r-1}W+\sum_{i=1}^{\lfloor\frac{r-1}{2}\rfloor}(r-2i)(x_{v_{a+1+i}}-x_{v_{\ell-b+1-i}})\\
&+(r-1)\left(\sum_{i=1}^{a+1}(x_{v_i}-x_{v_{\ell+1-i}})+\sum_{i=1}^{a}(x_{w_i}-x_{w_{\ell-i}})\right).
\end{align*}
Now by Lemma \ref{ab}(i--iii) and the hypothesis that  $x_{v_{a+1}}-x_{v_{\ell-b+1}}\leq x_{v_{a+2}}-x_{v_{\ell-b}}$, we have
\begin{align*}
&\rho(T)(x_{v_{a+2}}-x_{v_{\ell-b}})\\
<& \frac{2r-3}{r-1}W+\sum_{i=1}^{\lfloor\frac{r-1}{2}\rfloor}(r-2i)(x_{v_{a+2}}-x_{v_{\ell-b}})\\
&+(r-1)\left(\sum_{i=1}^{a+1}(x_{v_{a+1}}-x_{v_{\ell-a}})+\sum_{i=1}^{a}(x_{v_{a+1}}-x_{v_{\ell-a}})\right)\\
<&  \frac{2r-3}{r-1}W+\sum_{i=1}^{\lfloor\frac{r-1}{2}\rfloor}(r-2i)(x_{v_{a+2}}-x_{v_{\ell-b}})\\
&+(r-1)\left(\sum_{i=1}^{a+1}(x_{v_{a+1}}-x_{v_{\ell-b+1}})+\sum_{i=1}^{a}(x_{v_{a+1}}-x_{v_{\ell-b+1}})\right)\\
\leq &  \frac{2r-3}{r-1}W+\left((2a+1)(r-1)+\sum_{i=1}^{\lfloor\frac{r-1}{2}\rfloor}(r-2i)\right)(x_{v_{a+2}}-x_{v_{\ell-b}}).
\end{align*}
Thus
\[
\frac{2r-3}{r-1}W>\left(\rho(T)-\left((2a+1)(r-1)+\sum_{i=1}^{\lfloor\frac{r-1}{2}\rfloor}(r-2i)\right)\right)(x_{v_{a+2}}-x_{v_{\ell-b}}).
\]
By Lemma \ref{ab}(ii), $x_{v_{a+2}}-x_{v_{\ell-b}}>0$. By Lemma \ref{sum},
$\rho(T)>(2a+1)(r-1)+\sum_{i=1}^{\lfloor\frac{r-1}{2}\rfloor}(r-2i)$. Thus, $W>0$, a contradiction. It follows that $x_{v_{a+1}}-x_{v_{\ell-b+1}}> x_{v_{a+2}}-x_{v_{\ell-b}}>0$.
By  similar arguments as above and from the distance  eigenequations of $T$ at $v_{a+1}$ and $v_{\ell-b+1}$,
 we have
\begin{align*}
&\rho(T)(x_{v_{a+1}}-x_{v_{\ell-b+1}})\\
=&r(\sigma_{T}(A)-\sigma_{T}(B))+C+(r-1)x_{w_{\ell-b}}\\
%\leq&(r-2)(\sigma_{T}(A)-\sigma_{T}(B))+C+(r-2)\sum_{i=l-b+1}^{l-a-1}(x_{v_{i}}+x_{w_{i}})\\
<&r(\sigma_{T}(A)-\sigma_{T}(B))+C+(r-1)\sum_{i=\ell-b+1}^{\ell-a-1}(x_{v_{i}}+x_{w_{i}})\\
=& r(\sigma_{T}(A)-\sigma_{T}(B))+C+(r-1)(\sigma_{T}(A)-\sigma_{T}(B))\\
&-(r-1)\left(\sum_{i=1}^{a+1}(x_{v_{\ell+1-i}}-x_{v_i})+\sum_{i=1}^{a}(x_{w_{\ell-i}}-x_{w_i})\right)\\
=&\frac{r}{r-1}W-\frac{r}{r-1}C+W\\
&+(r-1)\left(\sum_{i=1}^{a+1}(x_{v_i}-x_{v_{\ell+1-i}})+\sum_{i=1}^{a}(x_{w_i}-x_{w_{\ell-i}})\right)\\
=&\frac{2r-1}{r-1}W-\frac{r}{r-1}C\\
&+(r-1)\left(\sum_{i=1}^{a+1}(x_{v_i}-x_{v_{\ell+1-i}})+\sum_{i=1}^{a}(x_{w_i}-x_{w_{\ell-i}})\right)\\
=&\frac{2r-1}{r-1}W+\frac{r}{r-1}\sum_{i=1}^{\lfloor\frac{r-1}{2}\rfloor}(r-2i)(x_{v_{a+1+i}}-x_{v_{\ell-b+1-i}})\\
&+(r-1)\left(\sum_{i=1}^{a+1}(x_{v_i}-x_{v_{\ell+1-i}})+\sum_{i=1}^{a}(x_{w_i}-x_{w_{\ell-i}})\right)\\
<&  \frac{2r-1}{r-1}W+\frac{r}{r-1}\sum_{i=1}^{\lfloor\frac{r-1}{2}\rfloor}(r-2i)(x_{v_{a+2}}-x_{v_{\ell-b}})\\
&+(r-1)\left(\sum_{i=1}^{a+1}(x_{v_{a+1}}-x_{v_{\ell-a}})+\sum_{i=1}^{a}(x_{v_{a+1}}-x_{v_{\ell-a}})\right)\\
<&  \frac{2r-1}{r-1}W+\frac{r}{r-1}\sum_{i=1}^{\lfloor\frac{r-1}{2}\rfloor}(r-2i)(x_{v_{a+2}}-x_{v_{\ell-b}})\\
&+(r-1)\left(\sum_{i=1}^{a+1}(x_{v_{a+1}}-x_{v_{\ell-b+1}})+\sum_{i=1}^{a}(x_{v_{a+1}}-x_{v_{\ell-b+1}})\right)\\
\leq&  \frac{2r-1}{r-1}W+\left((2a+1)(r-1)+\frac{r}{r-1}\sum_{i=1}^{\lfloor\frac{r-1}{2}\rfloor}(r-2i)\right)(x_{v_{a+1}}-x_{v_{\ell-b+1}}).
\end{align*}
Thus
\[
\frac{2r-1}{r-1}W>\left(\rho(T)-\left((2a+1)(r-1)+\frac{r}{r-1}\sum_{i=1}^{\lfloor\frac{r-1}{2}\rfloor}(r-2i)\right)\right)(x_{v_{a+1}}-x_{v_{\ell-b+1}}).
\]
By Lemma  \ref{sum},
$\rho(T)>(2a+1)(r-1)+\frac{r}{r-1}\sum_{i=1}^{\lfloor\frac{r-1}{2}\rfloor}(r-2i)$. So $W>0$ and then $\rho(T')>\rho(T)$.
\end{proof}

\section{Proof of Theorem \ref{t0}}

Now we are ready to give a proof to Theorem \ref{t0}.

%\begin{Theorem} \label{t}
%Let $T$ be a  hypertree of size at most three on $n$ vertices with $t$ edges of size three, where $1\leq t\leq \lfloor \frac{n-1}{2} \rfloor$.
%Then $\rho(T)\leq \rho (T(n,\lfloor\frac{t}{2}\rfloor, \lceil\frac{t}{2}\rceil))$ with equality if and only if $T\cong T(n,\lfloor\frac{t}{2}\rfloor, \lceil\frac{t}{2}\rceil)$.
%\end {Theorem}

\begin{proof}
%It is trivial if $t=\lfloor \frac{n-1}{2} \rfloor$. Suppose that $1\leq t<\lfloor \frac{n-1}{2} \rfloor$.

Let $T$ be a  hypertree of rank at most three on $n$ vertices with $k$ edges of size three that maximizes the distance spectral radius.

Let $x=x(T)$. Let  $\Delta$ be the maximum degree of $T$.

\noindent {\bf Claim 1.} $\Delta=2$.

Suppose that $\Delta\ge 3$. Then there is a vertex $u$ in $T$ with  $\mbox{deg}_T(u)=\Delta\geq3$. Thus $T$ consists of $\Delta$ maximal subhypertrees $T_1,\dots,T_{\Delta}$ containing a unique common vertex $u$ such that the degree of $u$ in any of $T_1,\dots,T_{\Delta}$ is one. Assume that $\sigma_T(T_1)\geq\sigma_T(T_2)$. Let $z$ be a pendant vertex in $T_2$.
%that is different from $u$.
Denote by $T'$ the hypergraph obtained from $T$ by moving the edge containing $u$ in $T_3$ from $u$ to $z$. Obviously, $T'$  is a hypertree of rank at most three on $n$ vertices with $k$ edges of size three. By Lemma \ref{3}, $\rho(T)<\rho(T')$, a  contradiction. Thus $\Delta=2$. This proves Claim 1.

Let $e$ be a pendant edge of $T$ at some vertex, say $v$.

\noindent {\bf Claim 2.} There is a vertex of degree one in any edge of size three of $T$.

Suppose that there is an edge of size three of $T$ in which every vertex has degree larger than one.
That is, there is an edge in $T$ of size  three, whose deletion yields three nontrivial components.
Choose such an edge $e_1=\{u_1,u_2,u_3\}$ so that $\min\{d_T(v, u_i): i=1,2,3\}$ is as large as possible.  Assume that $d_T(v,u_3)=\min\{d_T(v, u_i): i=1,2,3\}$.
Then $d_T(v,u_3)=d_T(v,u_i)-1$ for $i=1,2$.

For $i=1,2,3$, let $Q_i$ be the component of $T-e_1$  containing $u_i$.
Assume that  $\sigma_T(Q_1)\geq\sigma_T(Q_2)$.
By Claim 1,  $\Delta=2$. Let $e'$ be the unique edge in $Q_2$ containing $u_2$.

Suppose that $e'$ is of size two.
Let $T''$ be the hypertree obtained from $T$ by moving vertex  $u_3$  from $e_1$ to $e'$. Obviously, $T''$ is a hypertree of rank at most three on $n$ vertices with $k$ edges of size three. By  Lemma \ref{edge3}, $\rho(T)<\rho(T'')$, a  contradiction.
Thus $e'$ is of size three.

By the choice of $e_1$,  $T-e'$ has at most two nontrivial components, i.e., $T$ has  one pendant vertex, say $w$ in $e'$.
By Lemma \ref{entry}, we have $x_{u_{3}}+x_{u_2}-x_w>0$.
Let $T^*$ be the hypertree obtained from $T$ by moving the edge containing $u_3$ different $e_1$ from $u_3$ to $w$. Obviously, $T^*$  is a hypertree of rank at most three on $n$ vertices with $k$ edges of size three.
As we pass from   $T$ to $T^*$, the distance between a vertex of  $V(Q_3)\setminus \{u_3\}$  and a vertex of $V(Q_1)$  is increased by $1$,
the distance between a vertex of  $V(Q_3)\setminus \{u_3\}$  and a vertex of $V(Q_2)\setminus \{u_2,w\}$  is decreased by $1$, the distance between a vertex of  $V(Q_3)\setminus \{u_3\}$  and $u_3$ is increased by $2$, the distance between a vertex of  $V(Q_3)\setminus \{u_3\}$  and $w$ is decreased by $2$,  and the distance between any other vertex pair remains unchanged. Thus
\begin{align*}
\frac{1}{2}(\rho(T^*)-\rho(T))
&\geq\frac{1}{2}x^{\top}(D(\widetilde{T})-D(T))x\\
&=(\sigma_T(Q_3)-x_{u_3})(\sigma_T(Q_1)-(\sigma_T(Q_2)-x_{u_2}-x_w)+2x_{u_3}-2x_w)\\
&=(\sigma_T(Q_3)-x_{u_3})(\sigma_T(Q_1)-\sigma_T(Q_2)+2x_{u_3}+x_{u_2}-x_w)\\
&>0,
\end{align*}
so $\rho(T)<\rho(T^*)$, a  contradiction.
Thus,
there is a vertex of degree one in any edge of size three. This proves Claim 2.

By Claim 1, the maximum degree is two. By Claim 2, there is a vertex of degree one in any edge of size three of $T$. So
$T$ is a  loose path.
Now the result follows  if $k=\lfloor \frac{n-1}{2} \rfloor$. Suppose that $1\leq k<\lfloor \frac{n-1}{2} \rfloor$.

Suppose that there is an edge of size $3$, say  $e_2=\{w_1,w_2,w_3\}$ with $\mbox{deg}_T(w_1)=\mbox{deg}_T(w_2)=2$. For $i=1,2$, let $F_i$ be the component of $T-e_2$  containing $w_i$, respectively.

\noindent {\bf Claim 3.}  One of $F_1$ or $F_2$ has only edges of size three.

Otherwise, both $F_1$ and $F_2$ have at least one edge of size two.
Assume that  $\sigma_T(F_1)\geq\sigma_T(F_2)$.
Let $e^*$ be an edge in $F_2$ of size two and $T^{**}$ be the hypertree obtained from $T$ by moving vertex  $w_3$  from $e_2$ to $e^*$.
Obviously, $T^{**}$
is a hypertree of rank at most three on $n$ vertices with $k$ edges of size three.
By  Lemma \ref{edge3}, $\rho(T)<\rho(T^{**})$, a  contradiction.  So Claim 3 follows.

Recall  that $T$ is a  loose path. By Claim 3,  all vertices in the edges of size two induce an ordinary path, so
$T\cong T(n,a,b)$, where $0\leq a\leq b$ and $a+b=k$. Suppose that $b\geq a+2$, then by Lemma \ref{ab-2}, we have $\rho (T(n,a+1,b-1))>\rho (T(n,a,b))$, a contradiction. It follows that $T\cong T(n,\lfloor\frac{k}{2}\rfloor, \lceil\frac{k}{2}\rceil)$.
\end{proof}

\vspace{3mm}

\noindent {\bf Acknowledgement.}  This work was supported by the National Natural Science Foundation of China (No.~12071158) and the Youth Innovative Talent
Project of Guangdong Province of China (No. 2020KQNCX160).

\end{document}